\documentclass[a4paper,10pt,twoside]{amsart}
\usepackage[english]{babel}
\usepackage[utf8]{inputenc}
\usepackage[a4paper,inner=2.2cm,outer=2.2cm,top=4cm,bottom=4cm,pdftex]{geometry}
\usepackage{fancyhdr}
\usepackage{color}
\usepackage{bold-extra}
\usepackage{mathrsfs}

\usepackage{comment}
\usepackage{graphics}
\usepackage{aliascnt}
\usepackage[pdftex,citecolor=green,linkcolor=red]{hyperref}
\usepackage{mathtools}

\usepackage{amsmath}
\usepackage{amsfonts}
\usepackage{amssymb}
\usepackage{amsthm}
\usepackage{enumerate}

\newtheorem{theorem}{Theorem}[section]

\newtheorem{lemma}[theorem]{Lemma}
\newtheorem{prop}[theorem]{Proposition}
\theoremstyle{definition}

\newtheorem{rem}[theorem]{Remark}

\numberwithin{equation}{section}

\newcommand\R{\mathbb{R}}
\newcommand\Z{\mathbb{Z}}
\newcommand\N{\mathbb{N}}
\newcommand\C{\mathbb{C}}

\newcommand\eps{\varepsilon}
\renewcommand{\Re}{\textnormal{Re}}
\renewcommand{\Im}{\textnormal{Im}}

\newcommand{\sump}{\mathop{{\sum}^{\raisebox{-3pt}{\makebox[0pt][l]{$*$}}}}}
\newcommand{\sumflat}{\mathop{{\sum}^{\raisebox{-3pt}{\makebox[0pt][l]{$\flat$}}}}}

\usepackage{tikz}
\usetikzlibrary{calc}
\newlength{\lyxlabelwidth}      
	{\settowidth{\lyxlabelwidth}{#2}
		\begin{description}[font=\normalfont,style=sameline,
			leftmargin=\lyxlabelwidth,#1]}
	{\end{description}}

\begin{document}

\title{On the Real Zeroes of Half-integral Weight Hecke Cusp Forms}

\author{Jesse J\"a\"asaari}
\address{Department of Mathematics and Statistics\\
  University of Turku\\
  20014 Turku, Finland}
  \email{jesse.jaasaari@utu.fi}
  
\subjclass[2010]{Primary 11F37; Secondary 11M06}
  
\begin{abstract}   
We examine the distribution of zeroes of half-integral weight Hecke cusp forms on the manifold $\Gamma_0(4)\backslash\mathbb H$ near a cusp at infinity. In analogue of the Ghosh--Sarnak conjecture for classical holomorphic Hecke cusp forms, one expects that almost all of the zeroes sufficiently close to this cusp lie on two vertical geodesics $\Re(s)=-1/2$ and $\Re(s)=0$ as the weight tends to infinity. We show that, for $\gg_\varepsilon K^2/(\log K)^{3/2+\varepsilon}$ of the half-integral weight Hecke cusp forms in the Kohnen plus subspaces with weight bounded by a large parameter $K$, the number of such "real" zeroes grows almost at the expected rate. We also obtain a weaker lower bound for the number of real zeroes that holds for a positive proportion of forms. One of the key ingredients is the estimation of averaged first and second moments of quadratic twists of modular $L$-functions. 
\end{abstract}
  
\maketitle
         
\section{Introduction}

\noindent Studying the distribution of zeroes of automorphic forms has attracted attention both historically and more recently. A classical result in the theory of holomorphic modular forms for the group $\mathrm{SL}_2(\mathbb Z)$ on the upper half $\mathbb H$ of the complex plane is the so-called valence formula. It states that the number of (properly weighted) zeroes of such a form $f$ with weight $k$ is asymptotically $k/12$ inside the closure of the fundamental domain $\mathcal D:=\mathrm{SL}_2(\mathbb Z)\backslash\mathbb H=\{z\in\mathbb H:\, -1/2\leq \Re(z) \leq 1/2,\, |z|\geq 1\}$. Interestingly, it turns out that on smaller scales the distribution of zeroes of different types of automorphic forms can vary drastically. For instance, Rankin and Swinnerton-Dyer \cite{Rankin-Swinnerton--Dyer1970} have shown that all the zeroes of holomorphic Eisenstein series lie on the arc $\{|z|=1\}$ and moreover these zeroes are uniformly distributed there as $k\longrightarrow\infty$. In contrast, powers of the modular discriminant have a single zero of high multiplicity at infinity. Another type of behaviour is displayed by Hecke cusp forms, which form a natural basis for the space of modular forms of a given weight. 

For the latter forms Rudnick \cite{Rudnick2005} showed that the equidistribution of zeroes inside $\mathcal D$, as the weight of the form tends to infinity, follows from the holomorphic analogue of the Quantum Unique Ergodicity conjecture (QUE) of himself and Sarnak \cite{Rudnick-Sarnak1994}, which is spelled out explicitly in \cite{Luo-Sarnak2003, Sarnak2001}. As this conjecture has since been shown to hold by Holowinsky and Soundararajan \cite{Holowinsky-Soundararajan2010}, Rudnick's equidistribution result is unconditional. Given this, it is natural to wonder about the finer distributional behaviour of the zeroes of Hecke cusp forms, e.g. their distribution within subsets of $\mathcal D$ that shrink as the weight grows.

Such questions were first explored by Ghosh and Sarnak \cite{Ghosh-Sarnak2012} who considered the distribution of zeroes in shrinking domains around the cusp of the modular surface $\mathcal D$ at infinity. To be precise, they studied zeroes inside the Siegel sets $\mathcal D_Y:=\{z\in\mathcal D:\,\Im(z)\geq Y\}$ with $Y\longrightarrow\infty$ sufficiently fast along with $k\longrightarrow\infty$. Such a set can be regarded as a shrinking ball around the cusp at infinity as the hyperbolic area of $\mathcal D_Y$ equals $1/Y$ and so tends to zero as the weight tends to infinity. When $Y\gg k$ it is easy to see that there are no zeroes in $\mathcal D_Y$ except for the isolated simple zero at the cusp. Ghosh and Sarnak observed that, although the number of zeroes inside $\mathcal D_Y$ is proportional to the area of the domain, the statistical behaviour of the zeroes was very different from the uniform distribution when $\sqrt{k\log k}\ll Y\ll k$. Indeed, they observed that equidistribution of zeroes should not happen inside these sets and, based on  numerical evidence and a random model, were led to conjecture that almost all of the zeroes inside $\mathcal D_Y$, for $Y$ in the same range as before, concentrate on the half-lines $\Re(s)=-1/2$ and $\Re(s)=0$. They termed such zeroes to be "real" as the cusp form is itself real-valued on these lines\footnote{In the same paper they also considered zeroes on the arc $\{|z|=1\}$ where $z^{k/2}f(z)$ is real-valued.}. Ghosh and Sarnak conjectured that $50\%$ of the zeroes in these shrinking Siegel sets should lie on the line $\Re(s)=-1/2$ and likewise $50\%$ on the line $\Re(s)=0$. Furthermore, they obtained some results in this direction. In the end they were able to produce $\gg_\varepsilon (k/Y)^{1/2-1/40-\varepsilon}$ such real zeroes in the range $\sqrt{k\log k}\ll Y<k/100$. The exponent was later improved to $1/2-\eps$ by Matom\"aki \cite{Matomaki2016}. Producing real zeroes on the individual lines $\Re(s)=-1/2$ and $\Re(s)=0$ is more challenging, but Ghosh and Sarnak succeeded, the current best results, which are of polynomial growth, are again due to Matom\"aki \cite{Matomaki2016}. Related to this Lester, Matom\"aki, and Radziwi\l\l$\,$\cite{Lester-Matomaki-Radziwill2018} have also obtained some further results that we shall discuss below in more detail.  

One may of course speculate that similar phenomenon holds also for other types of automorphic forms. The analogue of the holomorphic QUE conjecture is known for half-integral weight Hecke cusp forms thanks to the work of Lester and Radziwi{\l\l} \cite{Lester-Radziwill2020} under the Generalised Riemann Hypothesis (GRH). Similarly to Rudnick's work this implies the equidistribution of zeroes inside the fundamental domain $\Gamma_0(4)\backslash\mathbb H$, where $\Gamma_0(4)$ is the Hecke congruence subgroup of level $4$, of course conditionally on GRH. Given this, it is natural to wonder whether the aforementioned results concerning the small scale distribution of zeroes generalise to these forms. The goal of the present paper is to address this as it seems that these questions have not been explored previously for half-integral weight Hecke cusp forms.  

This is not that surprising as the half-integral weight situation has features that are not present in the setting of classical holomorphic cusp forms. Indeed, Ghosh and Sarnak (and subsequent works) exhibited a relationship between real zeroes of integral weight Hecke cusp forms and sign changes of their Fourier coefficients. We shall do the same in setting of half-integral weight Hecke cusp forms, but here there are certain key differences that make adapting the methods used in the integral weight setting challenging. In the classical case the methods of \cite{Ghosh-Sarnak2012, Matomaki2016, Lester-Matomaki-Radziwill2018} used to study the behaviour of Fourier coefficients rely in a fundamental way on their multiplicativity. However, the Fourier coefficients of half-integral weight Hecke cusp forms lack this property, except at squares. Because of this the methods used in the previous works are not directly applicable in our setting and we have to use different tools to investigate the distribution of real zeroes.

The fundamental domain $\mathcal F:=\Gamma_0(4)\backslash\mathbb H$ is taken to be the domain in the upper half-plane bordered by the vertical lines $\sigma=\pm\frac12$ and circles\footnote{Here $B(x,r)$ denotes the circle with radius $r$ centred at $x$.} $B(-\frac13,\frac13)$, $B(\frac15,\frac15)$, and $B(\frac38,\frac18)$. Notice that this fundamental domain has three cusps at $\infty$, $0$, and $\frac12$ (these have widths $1$, $4$, and $1$, respectively). 

\begin{center}
\begin{tikzpicture}[scale=4]
\draw[very thick, fill=gray!30] (.5,1.5) -- (.5,0) arc (0:98.21:.125) arc (37:180:.2) arc (0:120:.33333) -- (-.5,.8)node[left]{$\delta_1$} -- (-.5,1.5)  ;
\draw[-latex] (-1.2,0) -- (1.2,0)node[below]{Re};
\draw[-latex] (0,-.2) -- (0,.8)node[right]{$\delta_2$} -- (0,1.7)node[right]{Im};
\path(-1,0) --node[below right, pos=.5]{0} (1,0);
\draw(-.5,.02)--(-.5,-.02)node[below]{$-\frac{1}{2}$}(-.33333,.02)--(-.3333,-.02)node[below]{$-\frac{1}{3}$}(.2,.02)--(.2,-.02)node[below]{$\frac{1}{5}$}(.375,.02)--(.375,-.02)node[below]{$\frac{3}{8}$}(.5,.02)--(.5,-.02)node[below]{$\frac{1}{2}$};
\draw[densely dashed] (-.66667,0) arc (180:120:.3333);
\draw[densely dashed] (.5,0) arc (0:180:.125);
\draw[densely dashed] (.4,0) arc (0:180:.2);
\draw[densely dashed] (-.5,0) -- (-.5,.2886751346);
\end{tikzpicture}
\end{center}

\begin{center}
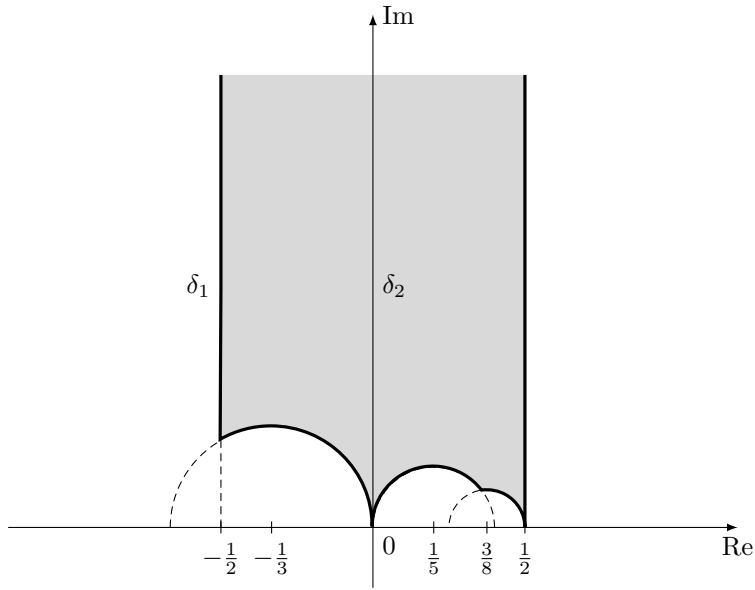
Figure 1. Fundamental domain for the action of the group $\Gamma_0(4)$ on $\mathbb H$.
\end{center}
\noindent It is well-known that any half-integral weight Hecke cusp form may be normalised to have real Fourier coefficients\footnote{Indeed, the numbers $c_g(n)n^{k/2-1/4}$ lie in the field generated over $\mathbb Q$ by the Fourier coefficients of its Shimura lift, which is a level $1$ Hecke eigenform of weight $2k$, so these numbers are real and algebraic (see \cite[Proposition 4.2]{Kumar-Purkait2014} and also the remarks before \cite[Theorem 1]{Kohnen-Zagier1981}).} $c_g(n)$ and we impose this normalisation throughout the paper.

Analogous to the integral weight setting, we study the real zeroes, that is zeroes on the two geodesic segments
\[ 
\delta_1:=\left\{s\in\mathbb C:\,\Re(s)=-\frac12\right\} \qquad \text{and} \qquad \delta_2:=\left\{s\in\mathbb C:\,\Re(s)=0\right\}
\] on which the cusp form takes real values in the normalisation above. Our first main result establishes that there are almost the expected amount (our result is optimal up to a power of the logarithm) of real zeroes for many half-integral weight Hecke cusp forms. Throughout the article, let $k$ be a positive integer. We write $S_{k+\frac12}(4)$ for the space of half-integral weight cusp forms of weight $k+\frac12$ and level $4$. Also, we denote by $S_{k+\frac12}^+(4)\subset S_{k+\frac12}(4)$ the Kohnen plus subspace and let $B^+_{k+\frac12}$ denote a fixed Hecke eigenbasis for $S_{k+\frac12}^+(4)$. Note that for half-integral weight cusp forms we cannot normalise the coefficient $c_g(1)$ to be equal to one without losing the algebraicity of the Fourier coefficients. This means that, unlike in the integral weight case, there is no canonical choice for $B_{k+\frac12}^+$, but this causes no problems for us. We write $\mathcal Z(g):=\{z\in\mathcal F:\,g(z)=0\}$ for the set of zeroes of $g\in S_{k+\frac12}^+(4)$ and let $\mathcal F_Y:=\{z\in\mathcal F\,:\,\Im(z)\geq Y\}$. Finally, set\footnote{Here, and throughout the paper, the notation $\ell\sim L$ means that $\ell\in[L,2L]\cap\mathbb N$. This should not be confused with the asymptotic notation $f\sim g$ for functions $f$ and $g$ used also in the present article.}
\[\mathcal S_K:=\bigcup_{k\sim K}B_{k+\frac12}^+\]
and note that the cardinality of this set is 
\[\sum_{k\sim K}\# B_{k+\frac12}^+=\sum_{k\sim K}\#\mathcal B_k\sim \frac{K^2}4
\]
as $\#\mathcal B_k\sim k/6$, where $\mathcal B_k$ is the Hecke eigenbasis for the space of holomorphic cusp forms of weight $2k$ and level one. 

With these notations our first main result may be stated as follows.

\begin{theorem}\label{main-theorem}
Let $K$ be a large parameter, $\varepsilon>0$ be an arbitrarily small fixed number, and $j\in\{1,2\}$. Then for $\gg_\varepsilon K^2/(\log K)^{3/2+\varepsilon}$ of the forms $g\in\mathcal S_K$ we have 
\[\#\{\mathcal Z(g)\cap\delta_j\cap\mathcal F_Y\}\gg\frac K{Y}\left(\log K\right)^{-23/2}  \]
for $\sqrt{K\log K}\leq Y\leq K^{1-\delta}$ with any small fixed constant $\delta>0$. 
\end{theorem} 

\begin{rem}
Here we have not tried to optimise the powers of logarithm. The exponents present are conveniently chosen so to that the various exponents that appear in the proof are simple fractions.  
\end{rem}

\noindent As indicated above, the methods used in the integral weight case are hard to implement to our setting. However, building on the multiplicativity of the Fourier coefficients at squares, we show that for positive proportion of forms it is also possible to get some real zeroes.

\begin{theorem}\label{second-main-theorem}
Let $K$ be a large parameter, $\eps>0$ be an arbitrarily small fixed number, and $j\in\{1,2\}$. Then for at least $(1/2-\eps)\#\mathcal S_K$ of the forms $g\in\mathcal S_K$ we have 
\[\#\{\mathcal Z(g)\cap\delta_j\cap\mathcal F_Y\}\gg\sqrt{\frac K{Y}} \]
for $\sqrt{K\log K}\leq Y\leq K^{1-\delta}$ with any small fixed constant $\delta>0$. 
\end{theorem}

\noindent As far as the author is aware of, these are the first results concerning the small scale distribution of zeroes of Hecke cusp forms in the half-integral weight setting\footnote{However, for elements in certain canonical basis, see \cite{Folsom-Jenkins2016}.}. Recall that the domain $\mathcal F_Y$ contains $\asymp k/Y$ zeroes of $g\in B_{k+\frac12}^+$, at least conditionally on GRH. The above results give progress towards the conjecture\footnote{In the integral weight case a probabilistic model for this is given by Ghosh and Sarnak \cite[Section 6]{Ghosh-Sarnak2012} and the same model works also in the half-integral weight setting.} that the number of zeroes of $g\in B^+_{k+\frac12}$ is $\gg k/Y$ on $\delta_j\cap \mathcal F_Y$ for both of the individual lines $\delta_j$ in the same range of $Y$ as before\footnote{One may refine this conjecture to state that each of the individual line segments $\delta_j\cap\mathcal F_Y$ contain $50\%$ of the zeroes.}. That is, apart from $\log$-powers our first result shows that we get the expected number of real zeroes for many, albeit not a positive proportion of, forms. The above theorems may be compared to results of Lester, Matom\"aki, and Radziwi{\l\l} \cite{Lester-Matomaki-Radziwill2018} who showed that in the setting of classical holomorphic Hecke cusp forms one has (for both $j\in\{1,2\})$
\[\#\{\mathcal Z(f)\cap\delta_j\cap \mathcal D_Y\}\gg_\eps\left(\frac k{Y}\right)^{1-\varepsilon}\]
for all $\varepsilon>0$ under the Generalised Lindel\"of Hypothesis, and unconditionally that for almost all forms one has
\[\#\left\{\mathcal Z(f)\cap\delta_j\cap\mathcal D_Y\right\}\asymp\frac kY \] 
in the same range of $Y$ as above. However, as indicated above, our techniques used to prove Theorem \ref{main-theorem} are very different compared to those used in the integral weight setting. 

\section{The Strategy}   

\noindent In this section we describe the main ideas that go into the proofs of the main theorems. Let $k$ be a positive integer and let $g$ be a Hecke cusp form of half-integral weight $k+\frac12$ and level $4$ that belongs to the Kohnen plus subspace.  Every such $g$ has a Fourier expansion
\begin{align*}
g(z)=\sum_{\substack{n=1\\
(-1)^k n\equiv 0,1\,(\text{mod }4)}}^\infty c_g(n)n^{\frac k2-\frac14}e(nz)\end{align*}
with normalised real Fourier coefficients $c_g(n)$, which encode arithmetic information. For instance, Waldspurger's formula shows that for fundamental discriminants $d$ with $(-1)^kd>0$, $|c_g(|d|)|^2$ is proportional to the central value of an $L$-function, and so the magnitude of the $L$-function essentially determines the size of the coefficient $c_g(|d|)$. We shall use this fact repeatedly. 

Let us then explain how to exploit information about the Fourier coefficients $c_g(n)$ in order to produce real zeroes. By a steepest descent argument it turns out that on the half-lines $\sigma+iy$, with $\sigma\in\{-\frac12,0\}$ and $y>0$, the values $g(\sigma+iy_\ell)$ for $y_\ell:=(k-1/2)/4\pi\ell$ and $\ell\in[c_1,c_2k/Y]$, for some absolute constants $c_1$ and $c_2$, are essentially determined by the single Fourier coefficient $c_g(\ell)$. Morally this reduces finding zeroes on these lines (i.e., real zeroes) to studying sign changes\footnote{In reality we need something a bit stronger, but for the purposes of sketching the argument we pretend here that it suffices to find sign changes.} of Fourier coefficients $c_g(\ell)$. More precisely, we consider a dyadic subinterval $[X,2X]\subset[c_1,c_2k/Y]$ with $X\asymp k/Y$ and wish to show that $c_g(\ell)$ has many sign changes for many forms $g$ as $\ell$ traverses over $\mathbb N\cap[X,2X]$.

The question of finding sign changes has been considered by many authors following the works of Knopp--Kohnen--Pribitkin \cite{KKP2003} and Bruinier--Kohnen \cite{Bruinier-Kohnen2008}, the former of which showed that such forms have infinitely many sign changes. Subsequent works \cite{HKKL2012, Kohnen-Lau-Wu2013, Lau-Royer-Wu2016} showed that the sequence $\{c_g(n)\}_n$ exhibits many sign changes under suitable conditions. It is essential for our approach to obtain quantitative results that are uniform in the weight aspect.   

Due to the lack of multiplicativity of these coefficients, the methods used to produce sign changes \cite{Ghosh-Sarnak2012, Matomaki2016, Lester-Matomaki-Radziwill2018} are not readily available in our setting. For some special subsequences of coefficients, methods of multiplicative number theory are useful. For example, for a fixed positive squarefree integer $t$ the coefficients $c_g(tm^2)$ are (essentially) multiplicative. This can be used in the proof of Theorem \ref{second-main-theorem} as we now explain. 

For this we rely on the following identity relating Fourier coefficients of half-integral weight Hecke cusp form $g$ at certain arguments and the Fourier coefficients of its Shimura lift $f$ at primes:
\begin{align}\label{Fourier-coeff-rel}
c_g(|d|p^2)=c_g(|d|)\left(\lambda_f(p)-\frac{\chi_d(p)}{\sqrt p}\right).
\end{align}
This is a special case of \cite[Equation (2)]{Kohnen-Zagier1981}. Here $p$ is any prime, $d$ is a fundamental discriminant with $(-1)^kd>0$, $\lambda_f(p)$ are the Hecke eigenvalues of $f$, and $\chi_d$ is the unique real quadratic character of conductor $|d|$. 

To benefit from this we have the following auxiliary result concerning the size of $c_g(|d|)$. For $g\in S_{k+\frac12}^+(4)$ let us define
\begin{align}\label{Normalisation}
\alpha_g:=\frac{\Gamma\left(k-\frac12\right)}{2\cdot(4\pi)^{k-\frac12}\|g\|_2^2},
\end{align}   
where the inner product $\|g\|_2^2:=\langle g,g\rangle$ is defined in (\ref{Inner_product}). With this normalising factor the analogue of the Ramanujan-Petersson conjecture for the Fourier coefficients predicts that for a Hecke eigenform $g$ we have $\sqrt{\alpha_g}c_g(m)\ll_\varepsilon k^{-1/2}(km)^\varepsilon$ with the implied constant depending only on $\varepsilon>0$.

\begin{prop}\label{non-vanishing}
Let $K$ be a large parameter and $\varepsilon,\,\theta>0$ be arbitrarily small, but fixed. Then for any fixed odd fundamental discriminant $d$ with $|d|\ll 1$ we have $\sqrt{\alpha_g}|c_g(|d|)|> k^{-1/2-\theta}$ for at least $(1/4-\eps)\#\mathcal S_K$ of the forms $g\in\mathcal S_K$ as $K\longrightarrow\infty$.
\end{prop}
\noindent This follows by combining asymptotics for the mollified second and fourth moments of the Fourier coefficients $c_g(|d|)$ using the Cauchy--Schwarz inequality and removing the harmonic weights using an approach of Iwaniec and Sarnak \cite{Iwaniec-Sarnak2000} detailed in \cite{Kowalski-Michel1999, Balkanova-Frolenkov2021}. The relevant moments are connected to moments of $L$-functions via Waldspurger's formula and so the mollifier is constructed to counteract the large values of $L(1/2,f\otimes \chi_d)$. To be more specific, a Dirichlet series mollifier is a short linear form
\[\sum_{\ell\leq L}\frac{\lambda_f(\ell)\chi_d(\ell)x_\ell}{\sqrt \ell},\]
where $L>0$ is the length of the mollifier and real numbers $x_\ell\ll\log L$ are chosen so that the mollifier mimics the behaviour of $L(1/2,f\otimes\chi_d)^{-1}$. Essentially the optimal choice for the coefficients $x_\ell$ translates to mollifier behaving approximately as
\[\sum_{\ell\leq L}\frac{\mu(\ell)\lambda_f(\ell)\chi_d(\ell)}{\sqrt\ell}\left(1-\frac{\log\ell}{\log L}\right)\]

The moment results required are stated in two lemmas below. These have been obtained\footnote{Notice that the leading constants in the main terms in \cite[Theorem 3]{Iwaniec-Sarnak2000} should be halved.} by Iwaniec and Sarnak \cite[Theorem 3]{Iwaniec-Sarnak2000}. Throughout the text we set
\[\omega_f:=\frac{2\pi^2}{2k-1}\cdot\frac1{L(1,\text{sym}^2f)}\]
for $f\in S_{2k}(1)$ and write $\mathcal B_k$ for the Hecke eigenbasis of $S_{2k}(1)$, the space of holomorphic cusp forms of weight $2k$ and full level. It is useful to note that  
\begin{align}\label{weights-average}
\sum_{f\in \mathcal B_k}\omega_f\sim 1. 
\end{align} 
Write $\mathcal M_{f,d}$ for the Iwaniec--Sarnak mollifier\footnote{In our notation the length $L$ of the mollifier corresponds to the parameter $M$ in \cite{Iwaniec-Sarnak2000}.} described in \cite[p. 163]{Iwaniec-Sarnak2000}.
 
\begin{lemma}\label{second-moment}
Let $h$ be a smooth compactly supported function on $\mathbb R_+$ taking non-negative values. Then, for the Iwaniec--Sarnak mollifier of length $L\leq |d|^{-1}K(\log K)^{-20}$, there exists a small absolute constant $c>0$ so that uniformly in odd fundamental discriminants $d$ with $|d|\leq K^c$ we have
\[\sum_{k\in\Z}h\left(\frac{2k}K\right)\sum_{f\in \mathcal B_k}\omega_f L\left(\frac12,f\otimes\chi_d\right)\mathcal M_{f,d}\sim \frac K2\int\limits_0^\infty h(t)\,\mathrm d t.\]
\end{lemma}

\begin{lemma}\label{fourth-moment}
Let $h$ be a smooth compactly supported function on $\mathbb R_+$ taking non-negative values. Then, for the Iwaniec--Sarnak mollifier of length $L\leq |d|^{-1}K(\log K)^{-20}$, there exists a small absolute constant $c>0$ so that uniformly in odd fundamental discriminants $d$ with $|d|\leq K^c$ we have
\[\sum_{k\in\Z}h\left(\frac{2k}K\right)\sum_{f\in \mathcal B_k}\omega_fL\left(\frac12,f\otimes\chi_d\right)^2\mathcal M_{f,d}^2\sim K\left(\int\limits_0^\infty h(t)\,\mathrm d t\right)\left(1+\frac{\log |d|K}{\log L} \right).\]
\end{lemma}
\noindent Given these, in the beginning of Section $10$ we explain that for a fixed odd fundamental discriminant $d$ with $|d|\ll 1$, the number of forms $f\in\bigcup_{k\sim K} \mathcal B_{k}$ for which $\omega_f L(1/2,f\otimes\chi_d)> k^{-1-2\theta}$ is 
\[> \frac{1}{4}\left(1-\frac{\log |d|}{\log K}\right)\sum_{k\sim K}\#\mathcal B_k.\]
But as the constraint $\omega_f L(1/2,f\otimes\chi_d)>k^{-1-2\theta}$ is equivalent to $\sqrt{\alpha_g}|c_g(|d|)>k^{-1/2-\theta}$ when $f$ is the Shimura lift of $g$ by Waldspurger's formula, Proposition \ref{non-vanishing} follows. 

From this we may conclude the proof of Theorem \ref{second-main-theorem} as follows. In our argument the precise range of uniformity in $d$ does not matter as we need these asymptotics for fixed $d\ll 1$. As alluded above, to find real zeroes it suffices to produce sign changes. Here the idea is to study sign changes of $c_g(|d|m^2)$ with $|d|\ll 1$ fixed and $m$ traversing over odd natural numbers. We will show that for a positive proportion of the forms $g\in B_{k+\frac12}^+$ one gets a positive proportion of sign changes for the sequence $c_g(|d|m^2)$ as $m$ traverses over the odd integers in the dyadic interval $[X,2X]$ with $X\asymp\sqrt{K/|d|Y}$. This can be done by implementing arguments of Lester, Matom\"aki, and Radziwi{\l\l} \cite{Lester-Matomaki-Radziwill2018} to (essentially) study sign changes of the multiplicative function $m\mapsto c_g(|d|)^{-1}c_g(|d|m^2)$ using (\ref{Fourier-coeff-rel}). For this Proposition \ref{non-vanishing} provides important information about the size of $c_g(|d|)$ for a positive proportion of forms. This leads to the promised amount of sign changes for at least $(1/2-\eps)\#\mathcal S_K$ of the forms in $\mathcal S_K$ by using different choices of $d$ to treat forms $g\in B_{k+\frac12}^+$ depending on the parity of $k$ separately. 

\begin{rem}
As indicated above, for a fixed $d$ we are able to study sign changes along the individual sequences $c_g(|d|m^2)$ as $m$ varies. Ideally one would want to study the sign changes of the Fourier coefficients by varying both $m$ and $d$. Note e.g. that for different fundamental discriminants $d$ the sequences $\{|d|p^2\}_p$, with $p$ traversing over the primes, are disjoint, but unfortunately it seems very hard to understand how these sequences are entangled. 
\end{rem}
\noindent Concerning the first main theorem it turns out that there is another way to produce sign changes along the fundamental discriminants that is not based on the multiplicativity of the Fourier coefficients. Recall that to find zeroes on individual lines $\delta_j$ it suffices (essentially) to detect sign changes of $c_g(m)$ along odd integers. For this we rely on the Shimura correspondence which attaches to $g\in S^+_{k+\frac12}(4)$ a classical integral weight Hecke cusp form $f$ of weight $2k$ and full level. Throughout the article we normalise the Shimura lift so that its first Fourier coefficient equals one.

For detecting sign changes we use an approach, which builds upon \cite{Matomaki-Radziwill2015, Lester-Radziwill2021}. We aim to obtain sign changes of $c_g(|d|)$ along squarefree $d\equiv 1\,(\text{mod }4)$, $(-1)^kd>0$, with $|d|\sim X$ where now $X\asymp K/Y$. The idea is to divide the dyadic interval $[X,2X]$ into short intervals $[x,x+H]$ with $1\leq H\leq x$ chosen as small as possible in terms of the weight so that we are able to detect sign changes and show that for many half-integral weight forms $g$ it holds that for many $x\sim X$ such a short interval $[x,x+H]$ contains a sign change of $c_g(|d|)$. 

To detect a sign change of $c_g(|d|)$ along such $d$ with $(-1)^kd\in[x,x+H]$ it suffices to have
\[\left|\sumflat_{x\leq (-1)^kd\leq x+H}\,c_g(|d|)\right|<\sumflat_{x\leq (-1)^kd\leq x+H}\,\left|c_g(|d|)\right|,\]
where $\sumflat\,$ means summing over squarefree integers $d\equiv 1$ (mod $4$).

We also note that with the normalisation (\ref{Normalisation}) Waldspurger's formula (see (\ref{Waldspurger-rel}) below) takes the form
\begin{align}\label{connecting-normalisations}
\alpha_g|c_g(|d|)|^2=\omega_fL\left(\frac12,f\otimes\chi_d\right),
\end{align}
where the Shimura lift $f\in S_{2k}(1)$ is normalised so that $\lambda_f(1)=1$, for any fundamental discriminant $d$ with $(-1)^k>0$. This is verified in Section $5$. Throughout the text the letter $g$ is reserved for half-integral weight cusp forms and its Shimura lift is always denoted by the letter $f$. Because of (\ref{connecting-normalisations}) it is convenient to normalise by the factor $\sqrt{\alpha_g}$ and seek for the inequality 
\[\left|\sumflat_{x\leq (-1)^kd\leq x+H}\,\sqrt{\alpha_g}c_g(|d|)\right|<\sumflat_{x\leq (-1)^kd\leq x+H}\,\sqrt{\alpha_g}\left|c_g(|d|)\right|,\]
which naturally leads to a sign change of $c_g(|d|)$ on the interval $[x,x+H]$ as $\sqrt{\alpha_g}$ is a positive real number. 

The following two estimates\footnote{Sometimes we write $\#\mathcal X$ for the cardinality of a set $\mathcal X$.} form a bulk of the proof of the first main theorem. Let us temporarily assume $X\asymp K/Y$ and let $H$ be a parameter with $1\leq H\leq X$ to be specified later. Here we recall the assumption $\sqrt{K\log K}\ll Y\ll K^{1-\delta}$ from which it follows that $K^\delta\ll X\ll\sqrt{K/\log K}$.

\begin{prop}\label{Prop1}
We have 
\[\#\left\{g\in\mathcal S_K:\,\#\left\{x\sim X:\,\left|\sumflat_{x\leq (-1)^kd\leq x+H}\,\sqrt{\alpha_g}c_g(|d|)\right|\geq \sqrt H k^{-1/2}(\log K)^3\right\}\gg \frac X{(\log X)^3}\right\}\ll \frac{K^{2}}{(\log K)^3}.\]
\end{prop}

\begin{prop}\label{Prop2}
We have
\[\#\left\{g\in\mathcal S_K:\,\#\left\{x\sim X:\,\sumflat_{x\leq (-1)^kd\leq x+H}\,\sqrt{\alpha_g}\left|c_g(|d|)\right|\geq\frac H{k^{1/2}\log X}\right\}\gg \frac X{(\log X)^{5/2}}\right\}\gg_\varepsilon\frac{K^2}{(\log K)^{3/2+\varepsilon}}\]
for any $\varepsilon>0$.
\end{prop}

\noindent Indeed, it is easy to see that the first statement implies that apart from $O(K^2/(\log K)^3)$ of the forms $g\in\mathcal S_K$ one has the property that the inequality 
\[\left|\sumflat_{x\leq (-1)^kd\leq x+H}\,\sqrt{\alpha_g}c_g(|d|)\right|<\sqrt H k^{-1/2}(\log K)^3\]
holds for almost all $x\sim X$ with the exceptional set having measure $\ll X/(\log X)^3$.

Likewise, the latter statement says that for $\gg_\varepsilon K^2/(\log K)^{3/2+\varepsilon}$ forms $g\in\mathcal S_K$ we have that 
\[\sumflat_{x\leq (-1)^kd\leq x+H}\,\sqrt{\alpha_g}\left|c_g(|d|)\right|\geq\frac H{k^{1/2}\log X}\]
for $x\gg X/(\log X)^{5/2}$ of $x\sim X$. Choosing, say, $H=(\log X)^9$ and combining the previous two observations, it follows that for $\gg_\varepsilon K^2/(\log K)^{3/2+\varepsilon}$ of the forms $g\in\mathcal S_K$ one has that for $\gg X/(\log X)^{5/2}$ of the numbers $x\sim X$ we have that 
\[ 
\left|\sumflat_{x\leq (-1)^kd\leq x+H}\sqrt{\alpha_g}\left|c_g(|d|)\right|\pm\sum_{x\leq(-1)^kd\leq x+H}\sqrt{\alpha_g}c_g(|d|)\right|\gg\frac H{\sqrt K\log X}.\]
From these inequalities it is a simple matter to deduce a sign change of $c_g(|d|)$ with $d$ an odd fundamental discriminant in the short intervals $[x,x+H]$. We refer to Section $9$ for details. This leads to $\gg X/(\log X)^{5/2}H$ sign changes for the same number of forms as above, which consequently yields the promised number of real zeroes. 

Both of the propositions above rely crucially on the sharp estimates for certain moments of quadratic twists of modular $L$-functions. The relevant results are the content of the following two lemmas.

\begin{lemma}\label{second-moment-with-average}
Let $\phi$ and $h$ be smooth compactly supported functions on $\mathbb R_+$.  Then
\[\sum_{k\in\Z}h\left(\frac{k}K\right)\sum_{g\in B^+_{k+\frac12}}\alpha_g\sumflat_d\, |c_g(|d|)|^2\phi\left(\frac{(-1)^k d}X\right)=\frac{2XK}{3\pi^2}\widehat h(0)\widehat\phi(0)+O_\eps\left(KX^{1/2+\eps}\right) \]
for any $\eps>0$. 
\end{lemma}

\begin{lemma}\label{fourth-moment-with-average}
Let $\phi$ be a smooth compactly supported function on $\mathbb R_+$. Then
\begin{align*}
&\sum_{k\sim K}\sum_{g\in B^+_{k+\frac12}}\alpha_g^2\omega_g^{-1}\sumflat_{d}\, |c_g(|d|)|^4\phi\left(\frac{(-1)^k d}X\right)\ll XK\log(XK),
\end{align*}
where the implicit constant depends on the weight functions $h$ and $\phi$.
\end{lemma}


\noindent In the latter lemma and throughout the paper we write as an abuse of notation $\omega_g$ for $\omega_f$ when $f$ is the Shimura lift of $g$. 

Proofs of these results are reasonably standard. The first asymptotic formula is based on the half-integral weight variant of the Petersson formula for forms in the Kohnen plus subspace whereas the latter is a consequence of a large sieve inequality of Deshouillers and Iwaniec \cite{DI1982}. We remark that in the proof of Lemma \ref{second-moment-with-average} the summation over $k$ plays no essential role and neither does the summation over the fundamental discriminants in the proof of Lemma \ref{fourth-moment-with-average}. Actually with more work one could also evaluate the fourth moment in Lemma \ref{fourth-moment-with-average} asymptotically with main term $c\cdot XK\log(XK)$, where the constant $c$ depends only on the weight function $h$. It is worth mentioning here that in such computation a nice structural feature that the off-diagonal contribution arising from an application of Petersson's formula cancels part of the diagonal contribution from the application of the same formula is present. This gives a new instance of a similar phenomenon appearing in some prior works \cite{Blomer2004, Blomer2008, Khan2010}.

A few remarks are in order concerning the various averages in the preceding two lemmas and in particular to explain why they are required. Ideally we would like to evaluate\footnote{When $j=4$ an upper bound of the right order of magnitude suffices for our purposes, but here we focus on the asymptotic evaluation as this would be needed for the incorporation of a mollifier to get a result that holds for a positive proportion of forms.} the moments 
\begin{align*}
\sumflat_{(-1)^k d\sim X}|c_g(|d|)|^j
\end{align*}
for $j\in\{2,4\}$ uniformly in terms of the weight $k+\frac12$ of $g$. By Waldspurger's formula these reduce to asymptotic evaluation of averaged first and second moments of $L(1/2,f\otimes \chi_{d})$ with uniformity in $k$. More precisely, we are required to understand the asymptotic behaviour of the sums
\begin{align}\label{moments}
\sumflat_{(-1)^k d\sim X}L\left(\frac12,f\otimes\chi_{d}\right) \qquad\text{and}\qquad \sumflat_{(-1)^k d\sim X}L\left(\frac12,f\otimes\chi_{d}\right)^2,
\end{align}
where $f\in S_{2k}(1)$ is the Shimura lift of $g$.

Recall that the complexity of a moment problem is measured by the ratio between logarithm of the analytic conductor and logarithm of the family size. Denote this ratio by $r$. The situations where $r=4$ is the edge of current technology where one can hope to obtain an asymptotic formula with a power saving error term. However, usually in this case we often barely fail to produce such asymptotics and a deep input is typically required in the rare cases when an asymptotic formula can be obtained.  Note that in the latter sum in (\ref{moments}) the ratio between the logarithm of the conductor and the logarithm of the family size is greater than four. Moreover, in our applications the parameter $X$ will be small, $X\ll\sqrt K$. In this situation the summation range is too short to evaluate even the first moment.

However, one can remedy the situation by introducing additional averages. Averaging over $g\in B_{k+\frac12}^+$ brings us to the situation where the ratio $r$ is precisely $4$ in the second moment problem and in this case it is plausible that the methods from Li's recent breakthrough \cite{Li2024} (which builds upon \cite{Soundararajan-Young2010}) can be adapted to our setting to yield the expected asymptotics, but saving only a power of a logarithm in the error term. Moreover, a direct adaptation of this method cannot handle the introduction of a mollifier of length $\gg X^\eps$. For all these reasons we have added one more averaging over the weights, which brings us to the situation where $2<r<3$. 

\subsection{Possible extensions} 
It is natural to wonder whether we can improve the result of Theorem \ref{main-theorem} to hold for a positive proportion of forms using a mollifier. By Waldspurger's formula a natural choice for the mollifier to study sign changes of $c_g(|d|)$ would be a quantity $\mathcal M_g(d)$, which is a truncated Dirichlet series approximation for $L(1/2,f\otimes\chi_d)^{-1/2}$. A traditional approach would be to start by letting the mollifying factor for $c_g(|d|)$ to be of the form
\[ 
\left(\sum_{\ell\leq L}\frac{x_\ell \mu(\ell)\lambda_f(\ell)\chi_d(\ell)}{\sqrt \ell}\right)^2\]
for some real coefficients $x_\ell$ to guarantee its positivity which is crucial for detecting sign changes. For this mollifier we would need to choose the coefficients $x_\ell$ so that
\[ 
\sum_{\ell\leq L}\frac{x_\ell \mu(\ell)\lambda_f(\ell)\chi_d(\ell)}{\sqrt \ell}\approx L\left(\frac12,f\otimes\chi_d\right)^{-1/4},\]
but unfortunately it seems unclear how to do this effectively. Moreover, a crucial step in the proof of Proposition \ref{Prop1} would involve understanding sums of the form
\begin{align}\label{g-av}
\sum_{g\in B_{k+\frac12}^+}\alpha_g c_g(|d|)c_g(|d|+h)\mathcal M_g(d)\mathcal M_g(d+h).
\end{align}
In this case, after several applications of Hecke relations, studying (\ref{g-av}) reduces to understanding the sums
\[ 
\sum_{g\in B_{k+\frac12}^+}\alpha_g c_g(|d|)c_g(|d|+h)\lambda_f(\ell).\]  This sum can be made amenable for Lemma \ref{half-int-Petersson} by inverting the relation (\ref{Fourier-kertoimien-yhteys}) to write $c_g(|d|)\lambda_f(\ell)$ as a linear combination of Fourier coefficients of $g$ at various arguments (at least for squarefree $\ell$). On the other hand, now evaluating the mollified moments 
\begin{align*}
&\sum_{k\in\Z}h\left(\frac{k}K\right)\sum_{g\in B^+_{k+\frac12}}\alpha_g\sumflat_{d}\, |c_g(|d|)|^2\mathcal M_{g}(d)^2\phi\left(\frac{(-1)^k d}X\right), \\
& \sum_{k\in\Z}h\left(\frac{k}K\right)\sum_{g\in B^+_{k+\frac12}}\alpha_g^2\omega_g^{-1}\sumflat_{d}\, |c_g(|d|)|^4\mathcal M_{g}(d)^4\phi\left(\frac{(-1)^k d}X\right),
\end{align*}
and optimising them becomes technically much more challenging compared to the corresponding moments with the Iwaniec--Sarnak mollifier $\mathcal M_{f,d}$ \cite{Iwaniec-Sarnak2000}.

Related to this, again by a repeated application of Hecke relations, a key step towards evaluating these mollified moments is to evaluating the twisted analogues of the sums appearing in Lemmas \ref{second-moment-with-average} and \ref{fourth-moment-with-average}. Indeed, with some additional work one may show that there exists an absolute constant $c>0$ so that uniformity in $\ell_1,\,\ell_2\leq K^c$ one can obtain asymptotics for 
\begin{align*}
\sum_{k\in\Z}h\left(\frac{k}K\right)\sum_{g\in B^+_{k+\frac12}}\alpha_g\sumflat_d\, |c_g(|d|)|^2\lambda_f(\ell_1)\chi_d(\ell_2)\phi\left(\frac{(-1)^k d}X\right)
\end{align*}
and 
\begin{align*}
\sum_{k\in\Z}h\left(\frac{k}K\right)\sum_{g\in B^+_{k+\frac12}}\alpha_g^2\omega_g^{-1}\sumflat_{d}\, |c_g(|d|)|^4\lambda_f(\ell_1)\chi_d(\ell_2)\phi\left(\frac{(-1)^k d}X\right).
\end{align*}
However, it is hard to benefit from this due to lack of understanding on how choose the coefficients $x_\ell$ optimally in the mollifier. 

A more promising approach, originating from \cite{Soundararajan2009, Harper2013, Radziwill-Soundararajan2015}, would be to use an Euler product mollifier as in the work of Lester and Radziwi{\l\l} \cite{Lester-Radziwill2021}. This leads to major technical complications when computing the fourth moment of the Fourier coefficients. Nevertheless, it is likely that this approach can be pushed through and we will investigate this possibility in a future work\footnote{After the present manuscript was accepted for publication, a result that holds for a positive proportion of forms has been obtained by the author \cite{Jaasaari2026} using an Euler product mollifier.}. We emphasise that the purpose of the present article is to demonstrate that one is able to deduce highly non-trivial information about the zeroes of half-integral weight Hecke cusp forms using methods very different from those used in the integral weight case that do not produce strong results in the present setting.  

\subsection{Organisation of the article}

This paper is organised as follows. In Section $5$ we gather basic facts about half-integral weight modular forms and other auxiliary results we need. In Section $6$ we prove a result that reduces the question of finding real zeroes to studying sign changes of the Fourier coefficients. Lemmas \ref{second-moment-with-average} and \ref{fourth-moment-with-average} are proved in Sections $7$ and $8$, respectively. These are then used in Section $9$ to prove Propositions \ref{Prop1} and \ref{Prop2} from which Theorem \ref{main-theorem} is deduced in the same section. In Section $10$ Proposition \ref{non-vanishing} is proved and the proof of Theorem \ref{second-main-theorem} is completed by studying sign changes of Fourier coefficients using ideas from multiplicative number theory. 

\section{Acknowledgements}

\noindent This work was supported by the Emil Aaltonen Foundation, the Finnish Cultural Foundation, and the Engineering and Physical Sciences Research Council [grant number EP/T028343/1]. The author would like to thank Steve Lester and Kaisa Matom\"aki for useful discussions. He is also grateful to the anonymous referee for a careful reading of the paper and for numerous helpful comments and corrections that improved the presentation of this work. A part of this paper was written while the author was visiting the Mittag-Leffler Institute in Djursholm, Sweden, whose excellent working conditions are acknowledged. 

\section{Notations}

\noindent We use standard asymptotic notation. If $f$ and $g$ are complex-valued functions defined on some set, say $\mathcal D$, then we write $f\ll g$ to  signify that $|f(x)|\leqslant C|g(x)|$ for all $x\in\mathcal D$ for some implicit constant $C\in\mathbb R_+$. The notation $O(g)$ denotes a quantity that is $\ll g$, and $f\asymp g$ means that both $f\ll g$ and $g\ll f$. We write $f=o(g)$ if $g$ never vanishes in $\mathcal D$ and $f(x)/g(x)\longrightarrow 0$ as $x\longrightarrow\infty$. Moreover, we write $f\sim g$ if $f(x)/g(x)\longrightarrow 1$ as $x\longrightarrow\infty$. The letter $\varepsilon$ denotes a positive real number, whose value can be fixed to be arbitrarily small, and whose value can be different in different instances in a proof.  All implicit constants are allowed to depend on $\varepsilon$, on the implicit constants appearing in the assumptions of theorem statements, and on anything that has been fixed. When necessary, we will use subscripts $\ll_{\alpha,\beta,...},O_{\alpha,\beta,...}$, etc. to indicate when implicit constants are allowed to depend on quantities $\alpha,\beta,...$

We define $\chi_d(\cdot):=\left(\frac d\cdot\right)$, the Jacobi symbol, for all non-zero odd integers $d$. Let us also write $1_{m=n}$ for the characteristic function of the event $m=n$. Furthermore, $\Re(s)$ and $\Im(s)$ are the real- and imaginary parts of $s\in\mathbb C$, respectively, and occasionally we write $\sigma$ for $\Re(s)$. We write $e(x):=e^{2\pi ix}$. For a compactly supported smooth function $\phi$, we define its Fourier transform $\widehat\phi(y)$ by
\[\widehat\phi(y):=\int\limits_\mathbb R \phi(x)e(-xy)\,\mathrm d x. \]
The sum $\sum_{a\,(c)}^*$ means that the summation is over residue classes coprime to the modulus. Given coprime integers $a$ and $c$, we write $\overline a\,(\text{mod } c)$ for the multiplicative inverse of $a$ modulo $c$. As usual, $\Gamma$ denotes the Gamma function and $\mu$ denotes the M\"obius function. Finally, $\sumflat\,\,$ means we are summing over all odd fundamental discriminants. 

\section{Preliminaries}

\subsection{Half-integral weight forms}

\noindent The group $\mathrm{SL}_2(\mathbb R)$ acts on the upper half-plane $\mathbb H$ by $\gamma.z:=\frac{az+b}{cz+d}$, where $\gamma=\begin{pmatrix} a & b\\
c & d
\end{pmatrix}$ and $z=x+iy\in\C$. Let $\Gamma_0(4)$ denote the congruence subgroup consisting of matrices $\begin{pmatrix} a & b\\
c & d
\end{pmatrix}$ in $\mathrm{SL}_2(\mathbb Z)$ such that $c$ is divisible by $4$. 

Let $\theta(z):=\sum_{n=-\infty}^\infty e(n^2z)$ denote the standard theta function on $\mathbb H$. If $A=\begin{pmatrix} a & b\\
c & d
\end{pmatrix}\in\Gamma_0(4)$, we have $\theta(Az)=j(A,z)\theta(z)$, where $j(A,z)$ is the so-called theta multiplier. For an explicit formula for $j(A,z)$, see \cite[1.10]{Shimura1973}. Fix a positive integer $k$. Let $S_{k+\frac12}(4)$ denote the space of holomorphic cusp forms of weight $k+\frac12$ for the group $\Gamma_0(4)$. This means that a function $g:\mathbb H\longrightarrow\mathbb C$ belongs to $S_{k+\frac12}(4)$ if 
\begin{itemize}
\item $g(Az)=j(A,z)^{2k+1}g(z)$ for every $A\in\Gamma_0(4)$.
\item $g$ is holomorphic.
\item $g$ vanishes at the cusps. 
\end{itemize} 
Any such form $g$ has a Fourier expansion of the form
\begin{align}\label{Fourier-exp}
g(z)=\sum_{n=1}^\infty c_g(n)n^{\frac k2-\frac14}e(nz),
\end{align}
where $c_g(n)$ are the Fourier coefficients of $g$.

For $g,h\in S_{k+\frac12}(4)$, we define the Petersson inner product $\langle g,h\rangle$ to be
\begin{align}\label{Inner_product}
\langle g,h\rangle:=\int\limits_{\Gamma_0(4)\backslash\mathbb H}g(z)\overline h(z)y^{k+\frac12}\frac{\mathrm d x\,\mathrm d y}{y^2}.
\end{align}

\noindent For any odd prime $p$ there exists a Hecke operator $T(p^2)$ acting on the space of half-integral weight modular forms given by 
\[T(p^2)g(z):=\sum_{n=1}^\infty\left(c_g(p^2n)+\left(\frac{(-1)^k n}p\right)p^{k-1}c_g(n)+p^{2k-1}c_g\left(\frac n{p^2}\right)\right)e(nz).\]
Here we have used the convention that $c_g(x)=0$ unless $x\in\mathbb Z$. We call a half-integral weight cusp form a Hecke cusp form if $T(p^2)g=\gamma_g(p)g$ for all $p>2$ for some $\gamma_g(p)\in\mathbb C$. 

The Kohnen plus subspace $S_{k+\frac12}^+(4)\subset S_{k+\frac12}(4)$ consists of all weight $k+\frac12$ Hecke cusp forms whose $n^{\text{th}}$ Fourier coefficient vanishes whenever $(-1)^k n\equiv 2,3\,(\text{mod }4)$. This space has a basis consisting of simultaneous eigenfunctions of $T(p^2)$ for odd $p$. As $k\longrightarrow\infty$, asymptotically one third of half-integral weight cusp forms lie in the Kohnen plus space by dimension considerations. In this space Shimura's correspondence \cite{Shimura1973} between half-integral weight forms and integral weight forms is well-understood. 

Indeed, Kohnen proved \cite{Kohnen1982} that there exists a Hecke algebra isomorphism between $S^+_{k+\frac12}(4)$ and the space of level $1$ cusp forms of weight $2k$. That is, $S^+_{k+\frac12}(4)\simeq S_{2k}(1)$ as Hecke modules. Also recall that every Hecke cusp form $g\in S^+_{k+\frac12}(4)$ can be normalised so that it has real Fourier coefficients and throughout the article we assume that $g$ has been normalised in this way.

For a fundamental discriminant $d$ with $(-1)^kd>0$ we know that
\begin{align}\label{Fourier-kertoimien-yhteys}
c_g(n^2|d|)=c_g(|d|)\sum_{r|n}\frac{\mu(r)\chi_d(r)}{\sqrt r}\lambda_f\left(\frac nr\right),
\end{align}
where $\lambda_f(n)$ denotes the $n^{\text{th}}$ Hecke eigenvalue of the Shimura lift $f$ (see equation $(2)$ of \cite{Kohnen-Zagier1981}). In
particular, if $n$ is a prime $p$ this reduces to (\ref{Fourier-coeff-rel}).

The proof of our first main result uses the explicit form of Waldspurger's formula due to Kohnen and Zagier \cite{Kohnen-Zagier1981}.

\begin{lemma} 
For a Hecke cusp form $g\in S^+_{k+\frac12}(4)$ we have
\begin{align}\label{Waldspurger-rel}
|c_g(|d|)|^2=L\left(\frac12,f\otimes\chi_d\right)\cdot\frac{(k-1)!}{\pi^k}\cdot\frac{\langle g,g\rangle}{\langle f,f\rangle}
\end{align}
for each fundamental discriminant $d$ with $(-1)^kd>0$, where $f$ is a holomorphic modular form attached to $g$ via the Shimura correspondence, normalised so that $\lambda_f(1)=1$. 
\end{lemma}
Here for $f_1,f_2\in S_{2k}(1)$ the Petersson inner product, which is still denoted by $\langle f_1,f_2\rangle$, is defined to be
\[ 
\langle f_1,f_2\rangle:=\int\limits_{\mathrm{SL}_2(\mathbb Z)\backslash\mathbb H}f_1(z)\overline{f_2}(z)y^{2k}\frac{\mathrm d x\,\mathrm d y}{y^2}.\]
We write $\|f\|_2^2:=\langle f,f\rangle$.


We also remark that $L(1/2,f\otimes\chi_d)$ vanishes when $(-1)^kd<0$ due to the sign in the functional equation. It follows directly from (\ref{Waldspurger-rel}) that $L(1/2,f\otimes\chi_d)\geq 0$ otherwise. To verify the formula (\ref{connecting-normalisations}) it suffices the show that 
\begin{align}\label{omega_alpha} 
\frac{\Gamma\left(k-\frac12\right)}{2(4\pi)^{k-1/2}\langle g,g\rangle}\cdot\frac{(k-1)!}{\pi^k}\cdot\frac{\langle g,g\rangle}{\langle f,f\rangle}=\frac{2\pi^2}{2k-1}\cdot\frac1{L(1,\text{sym}^2f)}.
\end{align}
To see this, we use a well-known identity (which follows from \cite[(7)]{Nelson2011} and Legendre's duplication formula)
\[ 
\langle f,f\rangle=\frac{\Gamma(k)\Gamma\left(k-\frac12\right)(2k-1)}{2^{2k+1}\pi^{2k+3/2}}\cdot L(1,\text{sym}^2f)\]
from which the assertion follows immediately. 

\subsection{Tools from the theory of automorphic forms}
Our first tool is an approximate functional equation for $L(1/2,f\otimes\chi_d)$. The following is an easy modification of \cite[Lemma 5]{Radziwill-Soundararajan2015}. 
\begin{lemma}\label{lem:AFE}
Let $f$ be a holomorphic Hecke cusp form of weight $2k$ for the full modular group $\mathrm{SL}_2(\mathbb Z)$ and $d$ be a fundamental discriminant with $(-1)^kd>0$. Then
\begin{align*}
L\left(\frac12,f\otimes\chi_d\right)=2\sum_{m=1}^\infty\frac{\lambda_f(m)\chi_d(m)}{\sqrt m}V_k\left(\frac m{|d|}\right),
\end{align*}
where, for any $\sigma>1/2$,
\begin{align}\label{integral-rep-for-weight}
V_k(x):=\frac1{2\pi i}\int\limits_{(\sigma)}g(s)x^{-s}e^{s^2}\frac{\mathrm d s}s\quad\text{with}\quad g(s):=(2\pi)^{-s}\frac{\Gamma(s+k)}{\Gamma(k)}.
\end{align}
Furthermore, we have 
\begin{align*}\label{weight-function-asymptotics}
V_k(\xi)=1+O\left(\frac\xi k\right)
\end{align*}
as $\xi\longrightarrow 0$.

We also have the estimates
\begin{align*}
& V_k(\xi)\ll_A \left(\frac k\xi\right)^A, \\
&V_k^{(B)}(\xi)\ll_{A,B}\xi^{-B}\left(\frac k\xi\right)^A
\end{align*}
for any $A>0$ and integer $B\geq 0$.
\end{lemma}


\noindent One of the most important tools is a half-integral weight analogue for the Petersson formula. The following result for forms in the Kohnen plus subspace is \cite[Lemma 6]{Blomer-Corbett2022} with a slightly different normalisation. Recall the definition of the normalising factor $\alpha_g$ from (\ref{Normalisation}). 

\begin{lemma}\label{half-int-Petersson}
Let $k\geq 3$ be an integer. Let $m,n$ be positive integers with $(-1)^km,\,(-1)^kn\equiv 0,\,1\,(\text{mod }4)$. Then 
\begin{align*}
\sum_{g\in B_{k+\frac12}^+}\alpha_g c_g(m)c_g(n)=\frac13\left(1_{m=n}+2\pi e\left(-\frac{k+\frac12}4\right)\sum_c\frac{K_{k+\frac12}^+(m,n;c)}c J_{k-\frac12}\left(\frac{4\pi\sqrt{mn}}c\right)\right),
\end{align*} 
where for $m,\,n\in\Z,\,c\in\N$, and $\kappa\in\frac12\Z$ we define the modified Kloosterman sum as
\begin{align*}
K_\kappa^+(m,n;c):=\sump_{d\,(c)}\epsilon_d^{2\kappa}\left(\frac cd\right)e\left(\frac{md+n\overline d}c\right)\cdot\begin{cases}
0 & \text{if }\, 4\nmid c \\
2 & \text{if }\, 4|c,\,8\nmid c \\
1 & \text{if }\, 8|c.
\end{cases} 
\end{align*}
Here 
\[
\epsilon_d:=\begin{cases}
1 & \text{if }d\equiv 1\,(4) \\
i & \text{if }d\equiv 3\,(4).
\end{cases}
\]
\end{lemma}

\noindent Note that the sum $K_\kappa^+$ is $2$-periodic in $\kappa$ and it satisfies a Weil-type bound (for this see \cite{Waibel2018}, and for the general theory see \cite[Chapter 11]{Iwaniec-Kowalski2004})
\begin{align}\label{modified-Kloosterman-bound}
|K_\kappa^+(m,n;c)|\leq d(c)(m,n,c)^{1/2}c^{1/2},
\end{align}
where $d(c)$ is the ordinary divisor function. 

The large sieve is a key tool in the proof of Lemma \ref{fourth-moment-with-average}. We employ the following variant due to Deshouillers and Iwaniec.

\begin{lemma}\label{large-sieve}
Let $K\geq 1$, $N\geq 1/2$, and let $\{a_n\}$ be any sequence of complex numbers. Then we have
\[ 
\sum_{k\sim K}\sum_{f\in\mathcal B_k}\omega_f\left|\sum_{n\sim N}a_n\lambda_f(n)\right|^2\ll_\varepsilon\left(K+K^{-1}N^{1+\eps}\right)\sum_{n\sim N}|a_n|^2\]
for any $\varepsilon>0$.
\end{lemma} 
\noindent This follows from \cite[Theorem 2]{DI1982} when taking the different normalisation into account. To check that the bound above is consistent with the one in \cite{DI1982}, note that using the well-known identity
\[ 
\Gamma\left(k-\frac12\right)=\frac{(2k-2)!}{4^{k-1}(k-1)!}\cdot\sqrt\pi\]
and (\ref{omega_alpha}) we have
\begin{align*}
\omega_f&=\frac{\Gamma\left(k-\frac12\right)}{2(4\pi)^{k-1/2}\|f\|_2^2}\cdot\frac{(k-1)!}{\pi^k}\\
&=\frac{(2k-2)!}{4^{k-1}(k-1)!}\cdot\sqrt\pi\cdot\frac{(k-1)!}{2(4\pi)^{k-1/2}\pi^k\|f\|_2^2}\\
&=\frac{(2k-2)!}{(4\pi)^{2k-1}\|f\|_2^2},
\end{align*}
which shows that $(2k-1)\omega_f$ corresponds to the normalising factor in \cite[(1.28)]{DI1982} (when also taking into account the normalisation by $\|f\|_2^2$ which arises as in \cite{DI1982} the Hecke eigenbasis is orthonormal whereas in our case it is just orthogonal).

We make use of another large sieve inequality for the Fourier coefficients $\lambda_f(n)$. The following is a special case of \cite[Theorem 1]{Lau-Wu2008}.

\begin{lemma}\label{large-sieve-2}
Let $v\geq 1$ be a fixed integer and $2\leq P<Q\leq 2P$. Then
\[ 
\sum_{f\in \mathcal B_k}\left|\sum_{P<p\leq Q}\frac{\lambda_f(p^v)}p\right|^2\ll_v k\frac1{P\log P}+k^{10/11}\frac{Q^{v/5}}{(\log P)^2}.\]
\end{lemma}

\noindent We also recall the following result of Murty and Sinha \cite[Theorem 2]{Murty-Sinha2009}.

\begin{lemma}\label{Murty-Sinha}
Let $p$ be a prime. Then for any interval $[\alpha,\beta]\subset[-2,2]$ we have
\[ 
\frac{\#\{f\in \mathcal B_k:\,\lambda_f(p)\in[\alpha,\beta]\}}{\#\mathcal B_k}=\int\limits_\alpha^\beta\mathrm d \mu_p+O\left(\frac{\log p}{\log k}\right),\]
where
\[\mathrm d \mu_p:=\frac{p+1}\pi\cdot\frac{(1-x^2/4)^{1/2}}{\left(p^{1/2}+p^{-1/2}\right)^2-x^2}\,\mathrm d x.\]
\end{lemma}

\subsection{Other tools}

\noindent We also record the following well-known uniform estimate for the $J$-Bessel function \cite[(2.11'')]{Iwaniec-Luo-Sarnak2000}. For $\nu\geq 0$ and $x>0$, the $J_\nu$-Bessel function satisfies the bound 
\begin{align}\label{J-Bessel}
J_\nu(x)\ll\frac x{\sqrt{\nu+1}}\left(\frac{ex}{2\nu+1}\right)^\nu.
\end{align}

\noindent Another key tool is the Poisson summation formula.

\begin{lemma}\label{Poisson}
Let $f$ be a Schwartz function and $a$ be a residue class modulo $c$. Then
\[\sum_{n\equiv\,a\,(c)}f(n)=\frac1c\sum_n\widehat f\left(\frac nc\right)e\left(\frac{an}c\right),\]
where $\widehat f$ denotes the Fourier transform of $f$. Note that this reduces to the classical Poisson summation formula when $c=1$. 
\end{lemma}

\noindent We need two results from the theory of multiplicative functions. The first one is a special case of the Matom\"aki--Radziwi{\l\l} theorem \cite[Theorem 1]{Matomaki-Radziwill2016}.

\begin{lemma}\label{MR} (\cite[Lemma 3.2]{Lester-Matomaki-Radziwill2018})
Let $f:\mathbb N\longrightarrow[-1,1]$ be a multiplicative function. Then there exists an absolute constant $C>1$ such that, for any $2\leq y\leq X$,
\[ 
\left|\frac1y\sum_{x\leq n\leq x+y}f(n)-\frac1X\sum_{n\sim X}f(n)\right|\leq 2(\log y)^{-1/200}\]
for almost all $x\sim X$ with at most $CX(\log y)^{-1/100}$ exceptions.
\end{lemma}

\noindent Our final tool is a variant of Hal\'asz's theorem (see \cite[(6)]{Hall-Tenenbaum1991}).

\begin{lemma}\label{Halasz}
Let $f:\mathbb N\longrightarrow[-1,1]$ be a multiplicative function. Then we have
\[ 
\frac1X\sum_{n\leq X}f(n)\ll \exp\left(-\frac14\sum_{p\leq X}\frac{1-f(p)}p\right).\]
\end{lemma}
 
\section{Reduction to the study of Fourier coefficients}

\noindent Our approach to relate detecting real zeroes to properties of the Fourier coefficients follows the previous works \cite{Ghosh-Sarnak2012, Matomaki2016, Lester-Matomaki-Radziwill2018} in the setting of integral weight cusp forms. The main result in this direction is the following observation, which is a half-integral weight analogue for \cite[Theorem 3.1.]{Ghosh-Sarnak2012}.

\begin{prop}\label{Fourier-coeff}
Let $\alpha\in\{-\frac12,0\}$. Then there are positive constants $c_1,c_2$ and $\eta$ such that, for all integers $\ell\in]c_1,c_2\sqrt{k/\log k}[$ and all Hecke eigenforms $g\in S^+_{k+\frac12}(4)$, we have
\begin{align*}
\sqrt{\alpha_g}\left(\frac e{\ell}\right)^{\frac{k}2-\frac14}g(\alpha+iy_\ell)=\sqrt{\alpha_g}c_g(\ell)e(\alpha\ell)+O(k^{-1/2-\eta}),
\end{align*}
where $y_\ell:=(k-1/2)/4\pi\ell$ and the implicit constant in the error term is absolute. 
\end{prop}
\noindent The proof is the same as in \cite{Ghosh-Sarnak2012}. The only automorphic input used in the proof, besides the existence of the Fourier expansion, is Deligne's bound for the Hecke eigenvalues, an analogue of which is not known for half-integral weight forms. However, it turns out that the pointwise bound\footnote{Note that in general the estimate in \cite{Conrey-Iwaniec2000} only holds for squarefree $m$. However, in our case $g$ is a Hecke eigenform and in this case the estimate can be extended for arbitrary $m\in\mathbb N$, see \cite[Lemma 3.3.]{JLLRW2019}.} $c(m)\ll_\varepsilon m^{1/6+\varepsilon}$ of Conrey and Iwaniec \cite[Corollary 3]{Conrey-Iwaniec2000} suffices once the $k$-dependence is made explicit. By the preceding footnote we have 
\begin{align}\label{square-free-upper-bound}
\sqrt{\alpha_g}|c_g(m)|\ll_\eps m^\eps\sqrt{\alpha_g}|c_g(\tilde m)|,
\end{align}
where $\tilde m$ is the squarefree kernel of $m$ and the implied constant depends only on $\eps>0$. To estimate the latter factor on the right-hand side one combines the works of Young \cite{Young2017} and Petrow--Young \cite{Petrow-Young2019} to see that for any squarefree\footnote{Note that in \cite{Young2017} the conductor of the quadratic character was assumed to be odd, but the same result holds for even discriminants as well \cite[Appendix A]{Andersen-Duke2020}.} integer $q$ one has $L(1/2,f\otimes\chi_q)\ll_\eps (kq)^{1/3+\eps}$ uniformly in both $k$ and $q$. Together with computations of Mao \cite[Appendix 2]{Mao2007} and a slight extension of these that appeared in \cite[Section 9]{Blomer-Harcos2008} this leads to the following pointwise bound. 

\begin{lemma}\label{pointwise}
Let $g\in S_{k+\frac12}^+(4)$ be a Hecke eigenform. Then we have
\begin{align}\label{pointwise-bound}
\sqrt{\alpha_g}c_g(m)\ll_\varepsilon k^{-1/3+\varepsilon}m^{1/6+\varepsilon}
\end{align}
for any $\varepsilon>0$. 
\end{lemma}

\noindent Recall that here the most optimistic bound would be $\ll_\eps k^{-1/2}(km)^\eps$. 

\begin{proof} In our normalisation \cite[(9.1)]{Blomer-Harcos2008} states that 
\[ 
c_g(\tilde m)\ll (k\tilde m)^\varepsilon(4\pi)^{k/2+1/4}\langle g,g\rangle^{1/2}\cdot\Gamma\left(k+\frac12\right)^{-1/2}L\left(\frac 12,f\otimes\chi_{\tilde m}\right)^{1/2}.
\]
Now using (\ref{square-free-upper-bound}), the relation $\Gamma(z+1)=z\Gamma(z)$, the estimate $L(1/2,f\otimes\chi_q)\ll_\eps (kq)^{1/3+\eps}$, and recalling the definition of $\alpha_g$ in (\ref{Normalisation}), we have
\begin{align*}
\sqrt{\alpha_g}c_g(m)&\ll_\varepsilon m^\varepsilon\sqrt{\alpha_g}c_g(\tilde m)\\
&\ll_\varepsilon\frac{(km)^\varepsilon}{\langle g,g\rangle^{1/2}}\langle g,g\rangle^{1/2}\left(\frac{\Gamma\left(k-\frac12\right)}{\Gamma\left(k+\frac12\right)}\right)^{1/2}L\left(\frac 12,f\otimes\chi_{\tilde m}\right)^{1/2}\\
&\ll_\varepsilon (km)^\varepsilon k^{-1/2}(km)^{1/6}\\
&\ll_\varepsilon k^{-1/3+\varepsilon}m^{1/6+\varepsilon}. 
\end{align*}
Here we have also used the fact that as $g$ is a Hecke eigenform it generates an irreducible cuspidal automorphic representation\footnote{Here $\mathbb A$ denotes the adele ring of rational numbers.} of $\widetilde{GL}_2(\mathbb A)$, the metaplectic cover of $\mathrm{GL}_2(\mathbb A)$, and thus corresponds to one of the special forms (see the beginning of \cite[Section 7.1]{Mao2007}) for which \cite[(9.1)]{Blomer-Harcos2008} is valid.
\end{proof}
\noindent The idea behind Proposition \ref{Fourier-coeff} is that the function $\xi\mapsto \xi^{k}e^{-\xi}$ has a maximum at $\xi=k$ and is very localised there. We give a fairly detailed argument as the $k$-dependence in (\ref{pointwise-bound}) is crucial and also remark that if the exponent of $k$ in Lemma \ref{pointwise} would be large than $-1/12$, then the method used to obtain the required approximation (Proposition \ref{Fourier-coeff}) would not succeed. This stems from treating the error term coming from Lemma \ref{I-approx} applied to the partial sum $\Phi_g^{(2)}$. 

Let us set
\[I_s(y):= y^{(s-1)/2}e^{-y}\]
for any $y>0$ and $s\in\mathbb C$. The following auxiliary result will be useful.

\begin{lemma}\label{I-approx}(\cite[Lemma 2.3.]{Ghosh-Sarnak2012})
For $|h|\ll k^{2/3-\delta}$, with any $\delta>0$, we have 
\[I_{k+\frac12}\left(\frac k2-\frac14+h\right)=I_{k+\frac12}\left(\frac k2-\frac14\right)e^{-h^2/(k-1/2)}\left(1+O\left(k^{-3\delta}\right)\right).\]
\end{lemma}
\noindent With this at hand we are ready to prove the approximation of the values of $g$ at certain arguments. \\

\noindent \emph{Proof of Proposition \ref{Fourier-coeff}.} Throughout the proof, let $y$ be a real parameter with $\sqrt k\ll y<k/100$. Furthermore, let $\eps>0$ be arbitrarily small but fixed. We freely refer to the argument of Ghosh and Sarnak for details. Just by the Fourier expansion (\ref{Fourier-exp}) we have
\begin{align*}
\sqrt{\alpha_g}g(\alpha+iy)&=\sum_{m=1}^\infty \sqrt{\alpha_g}c_g(m)m^{\frac k2-\frac14}e(m\alpha)e^{-2\pi my}\\
&=(2\pi y)^{-\frac k2+\frac14}\Phi_g(k;\alpha,y),
\end{align*}   
where 
\begin{align*}
\Phi_g(k;\alpha,y):&=\sum_{m=1}^\infty \sqrt{\alpha_g} c_g(m)e(m\alpha)(2\pi my)^{\frac k2-\frac14}e^{-2\pi my}\\
&=\sum_{m=1}^\infty \sqrt{\alpha_g}c_g(m)e(m\alpha)I_{k+\frac12}(2\pi my).
\end{align*}
Let $1\leq \Delta\ll k$ be a parameter specified later. We decompose $\Phi_g(k;\alpha,y)$ into three pieces according to the size of $m$: 
\begin{align*}
\Phi(k;\alpha,y)&=\sum_{\substack{m\geq 1\\
2\pi my<\frac k2-\frac14-\Delta}}\sqrt{\alpha_g}c_g(m)e(m\alpha)I_{k+\frac12}(2\pi my)+\sum_{|2\pi my-\frac k2+\frac14|\leq \Delta}\sqrt{\alpha_g}c_g(m)e(m\alpha)I_{k+\frac12}(2\pi my)\\
&\qquad\qquad+\sum_{2\pi my>\frac k2-\frac14+\Delta}\sqrt{\alpha_g}c_g(m)e(m\alpha)I_{k+\frac12}(2\pi my)\\
&=:\Phi_g^{(1)}(k;\alpha,y)+\Phi_g^{(2)}(k;\alpha,y)+\Phi_g^{(3)}(k;\alpha,y),
\end{align*}
say.

Starting with $\Phi_g^{(2)}(k;\alpha,y)$, we shall apply Lemma \ref{I-approx}. Choosing $h=2\pi my-(k/2-1/4)$ and $\delta=1/6-\varepsilon$ (anticipating the choice $\Delta\asymp \sqrt{k\log k})$ we have
\[ 
I_{k+\frac12}(2\pi my)=I_{k+\frac12}\left(\frac k2-\frac14\right)e^{-|2\pi my-k/2+1/4|^2/(k-1/2)}\left(1+O\left(k^{-3\delta}\right)\right).\]
Using Lemma \ref{pointwise} the error term contributes to $\Phi_g^{(2)}(k;\alpha,y)$ the amount
\begin{align*}
&\ll_\eps k^{-1/3+\varepsilon}\left(\frac ky\right)^{1/6+\varepsilon}I_{k+\frac12}\left(\frac k2-\frac14\right)\left(\sum_{|2\pi my-(k/2-1/4)|\leq\Delta}1\right)k^{-3\delta}\\
&\ll_\eps k^{-2/3+\varepsilon}y^{-1/6}I_{k+\frac12}\left(\frac k2-\frac14\right)\left(1+\frac\Delta y\right)\ll k^{-2/3+\varepsilon}y^{-1/6}I_{k+\frac12}\left(\frac k2-\frac14\right)
\end{align*}
as we shall choose $\Delta\asymp\sqrt{k\log k}$. Our aim now is to show that $\Phi^{(1)}$ and $\Phi^{(3)}$ both give a smaller contribution.

We first concentrate on $\Phi_g^{(3)}$. Using the identity $I_{s_1}(t)=t^{(s_1-s_2)/2}I_{s_2}(t)$ with $s_1=k+\frac12$, $s_2=k+\frac56+2\eps$, and the estimate (\ref{pointwise-bound}) we have 
\begin{align*}
\Phi_g^{(3)}(g;\alpha,y)&\ll_\eps k^{-1/3+\varepsilon}\sum_{\substack{m\geq 1\\
2\pi my>\frac k2-\frac14+\Delta}}m^{1/6+\varepsilon}(2\pi my)^{-1/6+\varepsilon}I_{k+\frac56+2\varepsilon}(2\pi my)\\
&\ll_\eps y^{-1/6-\varepsilon}k^{-1/3+\varepsilon}\sum_{\substack{m\geq 1\\
2\pi my>\frac k2-\frac14+\Delta}}I_{k+\frac56+2\varepsilon}(2\pi my). 
\end{align*}
Note that the function $t\mapsto I_{k+\frac56+2\varepsilon}(t)$ achieves its maximum at $t=k/2-1/12+\varepsilon$. On the other hand, as $\Delta\geq 1$, the function $t\mapsto I_{k+\frac 56+2\varepsilon}(2\pi ty)$ is decreasing in the domain we are considering. Thus we may estimate the sum by an integral as
\begin{align*}
\Phi_g^{(3)}(k;\alpha,y)&\ll_\eps y^{-1/6-\varepsilon}k^{-1/3+\varepsilon}\left(\int\limits_{(k/2-1/4+\Delta)/2\pi y}^\infty (2\pi yt)^{k/2-1/12+\varepsilon}e^{-2\pi yt}\,\mathrm d t+I_{k+\frac56+2\varepsilon}\left(\frac k2-\frac14+\Delta\right) \right)\\
&\ll_\eps y^{-1/6-\varepsilon}k^{-1/3+\varepsilon}\left(\frac 1y\Gamma\left(\frac k2+\frac{11}{12}+\varepsilon,\frac k2-\frac14+\Delta\right)+I_{k+\frac56+2\varepsilon}\left(\frac k2-\frac14+\Delta\right)\right),
\end{align*}
where $\Gamma(s,x)$ is the incomplete Gamma function.

By using the Taylor expansion for $\log(1+x)$ we have 
\[I_{k+\frac56+2\varepsilon}\left(\frac k2-\frac14+\Delta\right)\ll_\eps k^{1/6+\varepsilon}I_{k+\frac12}\left(\frac k2-\frac14\right)e^{-\Delta^2/(k-1/2)}.\]
Indeed, we have
\begin{align*}
\frac{I_{k+\frac56+2\varepsilon}\left(\frac k2-\frac14+\Delta\right)}{I_{k+\frac12}\left(\frac k2-\frac14\right)}&=\left(\frac k2-\frac14+\Delta\right)^{1/6+\varepsilon}\left(1+\frac\Delta{k/2-1/4}\right)^{k/2-1/4}e^{-\Delta}\\
&\ll_\eps k^{1/6+\varepsilon}\exp\left(\Delta-\frac{\Delta^2}{k-1/2}\right)e^{-\Delta}\\
&\ll_\eps k^{1/6+\varepsilon}e^{-\Delta^2/(k-1/2)},
\end{align*}
where we have used the Taylor expansion of $\log(1+x)$ in the form
\begin{align*} 
1+\frac\Delta{k/2-1/4}&=\exp\left(\log\left(1+\frac\Delta{k/2-1/4}\right)\right)\\
&=\exp\left(\frac\Delta{k/2-1/4}-\frac{\Delta^2}{2(k/2-1/4)^2}+O\left(k^{-1/2}\right)\right).
\end{align*}
Here the last estimate follows as we are going to choose $\Delta\asymp\sqrt{k\log k}$. This gives the desired estimate.

Estimating the incomplete Gamma function using a result of Natalini and Palumbo (see \cite{Natalini-Palumbo2000}) as in \cite{Ghosh-Sarnak2012} with the choices $a=k/2+11/2+\eps$, $x=k/2-1/4+\Delta$, and $\sigma=1+(k/2-1/4)/\Delta+\eps'$, we have 
\begin{align*}\Gamma\left(\frac k2+\frac{11}{12}+\varepsilon,\frac k2-\frac14+\Delta\right)&\ll_\eps\left(1+\frac k\Delta+\eps\right)\left(\frac k2-\frac14+\Delta\right)^{\frac k2-\frac1{12}+\eps}e^{-\frac k2+\frac14-\Delta}\\
& \ll_\eps\frac k\Delta e^{-\Delta^2/(k-1/2)}k^{1/6+\varepsilon} I_{k+\frac12}\left(\frac k2-\frac14\right).
\end{align*}
\noindent Combining all the preceding estimates we have shown that
\begin{align*}
\Phi_g^{(3)}(g;\alpha,y)\ll_\eps k^{-1/6+\varepsilon}y^{-1/6-\varepsilon}I_{k+\frac12}\left(\frac k2-\frac14\right)e^{-\Delta^2/(k-1/2)}\left(1+\frac1y\cdot\frac k\Delta\right).  
\end{align*}
Choosing $\Delta=\sqrt{A(k-1/2)\log k}$ for a sufficiently large fixed constant $A>0$ the right-hand side of the previous estimate is
\[\ll_\eps k^{-1/6-A+\varepsilon}y^{-1/6}I_{k+\frac12}\left(\frac k2-\frac14\right).\]

\noindent For $\Phi_g^{(1)}$ note that 
\[\sqrt{\alpha_g}c_g(m)\ll_\eps k^{-1/3+\varepsilon}m^{1/6+\varepsilon}\ll_\eps k^{- 1/6+\varepsilon}y^{-1/6-\eps}.\]
In addition, the function $t\mapsto I_{k+\frac12}(2\pi ty)$ is strictly increasing in the interval we are considering, and so we may again approximate the sum by an integral as
\begin{align}\label{Phi_1-int}
\Phi_g^{(1)}(k;\alpha,y)\ll_\eps k^{-1/6+\varepsilon}y^{-1/6-\eps}\left(\int\limits_1^{(k/2-1/4-\Delta)/2\pi y}(2\pi yt)^{k/2-1/4}e^{-2\pi yt}\,\mathrm d t+I_{k+\frac12}(2\pi y)+I_{k+\frac12}\left(\frac k2-\frac14-\Delta\right)\right).
\end{align}
Now  by the definition of $I_s(y)$ we have
\[ 
I_{k+\frac12}(2\pi y)=I_{k+\frac12}\left(\frac k2-\frac 14\right)\left(\frac{2\pi ye}{k/2-1/4}\right)^{k/2-1/4}e^{-2\pi y},\]
which decays exponentially when $y<k/100$.

As above, using the Taylor series approximation of $\log(1-x)$ we have 
\[I_{k+\frac12}\left(\frac k2-\frac14-\Delta\right)\ll I_{k+\frac12}\left(\frac k2-\frac 14\right)e^{-\Delta^2/(k-1/2).}\]
The integral in (\ref{Phi_1-int}) is estimated by splitting it into two parts. Let $\Delta_1>\Delta$ be a parameter specified later, and write
\begin{align*}
\int\limits_1^{(k/2-1/4+\Delta)/2\pi y}(2\pi yt)^{k/2-1/4}e^{-2\pi yt}\,\mathrm d t=\left(\int\limits_1^{(k/2-1/4-\Delta_1)/2\pi y}+\int\limits_{(k/2-1/4-\Delta_1)/2\pi y}^{(k/2-1/4-\Delta)/2\pi y}\right)(2\pi yt)^{k/2-1/4}e^{-2\pi yt}\,\mathrm d t.
\end{align*}
By making the change of variables $t\mapsto (\frac k2-\frac14)t/2\pi y$ this simplifies into 
\begin{align*}
&\frac{\left(\frac k2-\frac14\right)^{k/2+3/4}}{2\pi y}\left(\int\limits_{2\pi y/(k/2-1/4)}^{1-\Delta_1/(k/2-1/4)}+\int\limits_{1-\Delta_1/(k/2-1/4)}^{1-\Delta/(k/2-1/4)}\right) t^{k/2-1/4}e^{-\left(\frac k2-\frac14\right)t}\,\mathrm d t\\
&=\frac{\left(\frac k2-\frac14\right)}{2\pi y}I_{k+\frac12}\left(\frac k2-\frac14\right)\left(\int\limits_{2\pi y/(k/2-1/4)}^{1-\Delta_1/(k/2-1/4)}+\int\limits_{1-\Delta_1/(k/2-1/4)}^{1-\Delta/(k/2-1/4)}\right) t^{k/2-1/4}e^{\left(\frac k2-\frac14\right)(1-t)}\,\mathrm d t.
\end{align*}
As the integrand is strictly increasing, the latter integral is
\begin{align*}
\ll \frac{\Delta_1-\Delta}k\left(1-\frac\Delta{k/2-1/4}\right)^{k/2-1/4}e^\Delta\ll \frac{\Delta_1-\Delta}{k}e^{-\Delta^2/(k-1/2)},
\end{align*}
where we have again used the Taylor expansion of $\log(1-x)$ in the latter estimate.

Similarly, estimating by absolute values and using the Taylor expansion the other part of the integral is $\ll e^{-\Delta_1^2/(k-1/2)}$ provided that $\Delta_1=o(k^{2/3})$, which will be satisfied as we shall choose $\Delta_1:=\sqrt{B(k-1/2)\log k}$ for a fixed constant $B>A+\frac12$. With this choice, gathering all the estimates, we have 
\begin{align*}
\Phi_g^{(1)}(k;\alpha,y)&\ll_\varepsilon k^{-1/6+\varepsilon}y^{-1/6-\eps}I_{k+\frac12}\left(\frac k2-\frac14\right)\frac ky\left(\frac{\Delta_1-\Delta}{k}e^{-\Delta^2/(k-1/2)}+e^{-\Delta_1^2/(k-1/2)}\right)\\
&\ll_\varepsilon k^{-1/6+\varepsilon}y^{-1/6-\eps}I_{k+\frac12}\left(\frac k2-\frac14\right)\frac ky\left(\frac{\sqrt{k\log k}}k k^{-A}+k^{-B}\right)\\
&\ll_\varepsilon k^{-1/6-A+\varepsilon}y^{-1/6-\eps}I_{k+\frac12}\left(\frac k2-\frac14\right),
\end{align*}
where the last step follows from recalling that $B>A+1/2$ and choosing $A$ to be sufficiently large.

In total we have shown that 
\begin{align*}
&\sqrt{\alpha_g}g(\alpha+iy)\\
&=(2\pi y)^{-\frac k2+\frac14}I_{k+\frac12}\left(\frac k2-\frac14\right)\bigg(\sum_{\substack{m\geq 1\\
|2\pi my-(k/2-1/4)|\leq\Delta}}\sqrt{\alpha_g}c_g(m)e(m\alpha)e^{-|2\pi my-(k/2-1/4)|^2/(k-1/2)}\\
&\qquad\qquad\qquad\qquad\qquad\qquad\qquad\qquad\qquad\qquad+O_\eps\left(k^{-2/3+\eps}y^{-1/6}\right)\bigg).
\end{align*}
Now observe that for $y_\ell=(k-1/2)/4\pi\ell$,
\[\left|2\pi my_\ell-\left(\frac k2-\frac14\right)\right|\leq\Delta\quad\Longleftrightarrow\quad|m-\ell|\leq\frac{\Delta\ell}{k/2-1/4},\]
which forces $m=\ell$ in the light of the choice for $\Delta$ when $\ell\leq\sqrt{(k/2-1/4)/(2A\log k)}$. 

Hence, by recalling that $y\gg\sqrt k$,
\begin{align*}
\sqrt{\alpha_g}g(\alpha+iy_\ell)&=(2\pi y_\ell)^{-\frac k2+\frac14}I_{k+\frac12}\left(\frac k2-\frac14\right)\bigg(\sqrt{\alpha_g}c_g(\ell)e(\alpha\ell)+O\left(k^{-1/2-\eta}\right)\bigg)\\
&=\left(\frac{\ell}e\right)^{\frac k2-\frac14}\bigg(\sqrt{\alpha_g}c_g(\ell)e(\alpha\ell)+O\left(k^{-1/2-\eta}\right)\bigg)
\end{align*}
for, say, $\eta=1/4-\eps$ for any sufficiently small fixed $\eps>0$. This completes the proof. \qed

\section{Proof of Lemma \ref{second-moment-with-average}}

\noindent In this section we evaluate the average of the absolute squares of Fourier coefficients of half-integral weight Hecke cusp forms. 

Recall that the sum we aim to evaluate is
\[S_1:=\sum_{k\in\Z}h\left(\frac{k}K\right)\sum_{g\in B_{k+\frac12}^+}\alpha_g\sumflat_d\,|c_g(|d|)|^2\phi\left(\frac{(-1)^k d}X\right).\]
Executing the $g$-sum using Lemma \ref{half-int-Petersson} and estimating trivially this equals
\[ 
\sum_{k\in\Z}h\left(\frac kK\right)\sumflat_d\phi\left(\frac{(-1)^k d}X\right)\left(\frac13+O\left(\sum_{c=1}^\infty\frac{\left|K_{k+\frac12}(|d|,|d|;c)\right|}c\left|J_{k-\frac12}\left(\frac{4\pi|d|}c\right)\right|\right)\right).\]
The $J$-Bessel function is exponentially small when $c>100|d|/k$ by (\ref{J-Bessel}). Recalling that $|d|\leq X\ll\sqrt{K/\log K}$ and $k\asymp K$ this leads to
\begin{align*}
S_1&=\sum_{k\in\mathbb Z}h\left(\frac{k}K\right)\sumflat_d\phi\left(\frac{(-1)^kd}X\right)\left(\frac13+O\left(K^{-100}\right)\right)\\
&=\frac13\sum_{k\in\mathbb Z}h\left(\frac{k}K\right)\sumflat_d\phi\left(\frac{(-1)^kd}X\right)+O\left(XK^{-99}\right)
\end{align*}
using the Weil bound (\ref{modified-Kloosterman-bound}).

To evaluate the main term, we begin by detecting the condition that $d$ is squarefree by means of the identity
\begin{align}\label{squarefree-detection} 
\sum_{\substack{\alpha=1\\
\alpha^2|d}}^\infty\mu(\alpha)=\begin{cases}
1 & d\text{ is squarefree}\\
0 & \text{otherwise}
\end{cases}
\end{align}
This together with the Poisson summation (Lemma \ref{Poisson}) and noting that the compact support of $\phi$ restricts the $\alpha$-sum to $\alpha\ll \sqrt X$ gives
\begin{align*}
\sumflat_d\phi\left(\frac{(-1)^kd}X\right)&=\sum_{d\equiv 1\,(4)}\phi\left(\frac{(-1)^kd}X\right)\sum_{\alpha^2|d}\mu(\alpha) \\
&=\sum_{\substack{\alpha\ll\sqrt X\\
(\alpha,2)=1}}\mu(\alpha)\sum_{d\equiv\overline{\alpha^2}\,(4)}\phi\left(\frac{(-1)^k\alpha^2d}X \right)\\
&=\frac X{4}\sum_{\substack{\alpha\ll\sqrt X\\
(\alpha,2)=1}}\frac{\mu(\alpha)}{\alpha^2}\sum_{m\in\mathbb Z}\left(\int\limits_\R\phi(y)e\left(-\frac{mXy}{4(-1)^k\alpha^2}\right)\,\mathrm d y\right)e\left(\frac{\overline{\alpha^2}m}{4}\right).
\end{align*}
The main contribution arises when $m=0$. Noting that 
\[
\sum_{\substack{\alpha=1\\
(\alpha,2)=1}}^\infty\frac{\mu(\alpha)}{\alpha^2}=\frac8{\pi^2},\qquad\qquad\sum_{\substack{\alpha\gg\sqrt X\\
(\alpha,2)=1}}\frac{\mu(\alpha)}{\alpha^2}\ll_\eps X^{-1/2+\eps}\]
for any $\eps>0$, and summing the $k$-sum using Poisson summation we conclude that the $m=0$ contribution is given by 
\begin{align}\label{main-term}
\frac {2XK}{3\pi^2}\widehat h(0)\widehat\phi(0)+O_\eps\left(KX^{1/2+\eps}\right).
\end{align}

For the terms with $m\neq 0$, integrating by parts twice we see that the exponential integral is bounded by $\ll (\alpha^2/mX)^2$ and thus these terms contribute an amount
\begin{align*}
&\ll X\sum_{\alpha\ll\sqrt X}\frac1{\alpha^2}\sum_{m\neq 0}\frac{\alpha^4}{(mX)^2}\\
&\ll\frac 1X\sum_{\alpha\ll\sqrt X}\alpha^2\\
&\ll\sqrt X
\end{align*}
to the $d$-sum. Now estimating the $k$-sum trivially completes the proof. \qed

\section{Proof of Lemma \ref{fourth-moment-with-average}}

\noindent We now move to the proof of the second lemma. Recall that the sum we are considering is
\[
\sum_{k\sim K}\sum_{g\in B^+_{k+\frac12}}\alpha_g^2\omega_g^{-1}\sumflat_{d}\, |c_g(|d|)|^4\phi\left(\frac{(-1)^k d}X\right).
\]
Using the relation (\ref{connecting-normalisations}) our sum takes the form
\[\sum_{k\sim K}\sum_{f\in\mathcal B_k}\sumflat_{d}\,\omega_f L\left(\frac12,f\otimes\chi_{d}\right)^2\phi\left(\frac {(-1)^kd}X\right), \] 
and applying the approximate functional equation (Lemma \ref{lem:AFE}) bounds this from the above by 
\begin{align*}
\ll \sumflat_d\phi\left(\frac{|d|}X\right)\sum_{k\sim K}\sum_{f\in\mathcal B_k}\omega_f\left|\sum_{m\leq (k|d|)^{1+\varepsilon}}\frac{\lambda_f(m)\chi_d(m)}{\sqrt m}V_k\left(\frac m{|d|}\right)\right|^2+K^{-100}.
\end{align*}
Now applying Lemma \ref{large-sieve}, recalling that $|d|\ll X\ll\sqrt K$, and estimating the $d$-sum trivially yields an upper bound
\begin{align*} 
&\ll_\varepsilon\log (KX)\sumflat_d\phi\left(\frac{|d|}X\right)\left(K+K^{-1}(K|d|)^{1+\eps}\right) \\
& \ll XK\log (XK). 
\end{align*}
This completes the proof. \qed
 
\section{Proof of Theorem \ref{main-theorem}}

\noindent We first prove Theorem \ref{main-theorem} assuming the truth of Propositions \ref{Prop1} and \ref{Prop2}. Let $\alpha\in\{-\frac12,0\}$, $\varepsilon>0$ be any fixed small constant, and let $K$ be large positive parameter. Recall that as $g(\alpha+iy)$ is real-valued for these values of $\alpha$, Proposition \ref{Fourier-coeff} yields information on the zeroes inside the Siegel sets
\[\mathcal F_Y=\{z\in\Gamma_0(4)\backslash\mathbb H:\, \Im(z)\geq Y\}\]
with $c_1'\sqrt{k\log k}\leq Y\leq c_2'k$ for some positive constants $c_1'$ and $c_2'$. Let $\eta$ be as in Proposition \ref{Fourier-coeff}. It follows immediately from that result that if we can find numbers $\ell_1,\ell_2\in]c_1,c_2k/Y[$ so that 
\begin{align}\label{real-zero}
\sqrt{\alpha_g}c_g(\ell_1)e(\alpha\ell_1)<-k^{-\delta}<k^{-\delta}<\sqrt{\alpha_g}c_g(\ell_2)e(\alpha\ell_2)
\end{align}
for some $\delta<1/2+\eta$, then $g(z)$ has a zero $\alpha+iy$ with $y$ between $y_{\ell_1}$ and $y_{\ell_2}$. Observe that 
\begin{align*}
e(\alpha\ell)=\begin{cases}
1 & \text{if }\alpha=0\\
(-1)^\ell & \text{if }\alpha=-1/2
\end{cases}
\end{align*}
Hence, in order to find real zeroes on the line $\Re(s)=0$ it suffices (essentially) to detect sign changes among the Fourier coefficients whereas on the line $\Re(s)=-1/2$ one needs to find pairs $(\ell_1,\ell_2)$ with $\ell_i$ odd for which (\ref{real-zero}) holds. As we restrict to odd fundamental discriminants $d$ for which $(-1)^k d>0$, we automatically obtain real zeroes on both of the individual geodesic segments $\Re(s)=-1/2$ and $\Re(s)=0$. 

Remember that in order to detect sign changes along the sequence $d\equiv 1\,(\text{mod }4)$ with $d$ squarefree and $(-1)^kd>0$ in the short interval $[x,x+H]$, $x\sim X\asymp K/Y$, it suffices to have
\[\left|\sumflat_{x\leq (-1)^kd\leq x+H}\sqrt{\alpha_g}c_g(|d|)\right|<\sumflat_{x\leq (-1)^kd\leq x+H}\sqrt{\alpha_g}|c_g(|d|)|.\]

\noindent Choose $H=(\log X)^9$. With this choice it follows easily from Propositions \ref{Prop1} and \ref{Prop2} that for $\gg_\eps K^2/(\log K)^{3/2+\eps}$ of the forms $g\in\mathcal S_K$ we have that
\[ 
\left|\sumflat_{x\leq (-1)^kd\leq x+H}\sqrt{\alpha_g}|c_g(|d|)|\,\pm\sumflat_{x\leq (-1)^kd\leq x+H}\sqrt{\alpha_g}c_g(|d|)\right|\gg \frac H{\sqrt k \log X}\]
holds for $\gg X/(\log X)^{5/2}$ of the numbers $x\sim X$. We note that the contribution coming from the summands with $\sqrt{\alpha_g}|c_g(|d|)|\leq k^{-\delta}$ is trivially bounded by
\[\ll \sum_{\substack{x\leq (-1)^k d\leq x+H\\
\sqrt{\alpha_g} |c_g(|d|)|\leq k^{-\delta}}}\sqrt{\alpha_g}|c_g(|d|)|\leq 2HK^{-\delta}\]
and so we conclude, choosing $\delta=1/2+\eta/2$ for concreteness so that $\delta>1/2$, that for the same proportion of $g\in \mathcal S_K$ and $x\sim X$ we have
\[
\left|\sumflat_{\substack{x\leq (-1)^kd\leq x+H\\
\sqrt{\alpha_g}|c_g(|d|)|> k^{-\delta}}}\sqrt{\alpha_g}|c_g(|d|)|\,\pm\sumflat_{\substack{x\leq (-1)^kd\leq x+H\\
\sqrt{\alpha_g}|c_g(|d|)|> k^{-\delta}}}\sqrt{\alpha_g}c_g(|d|)\right|\gg \frac H{\sqrt k \log X}.\]
Now observe that this implies 
\begin{align*}
2\sumflat_{\substack{x\leq (-1)^kd\leq x+H\\
\sqrt{\alpha_g}c_g(|d|)> k^{-\delta}}}\sqrt{\alpha_g}|c_g(|d|)|&=\left|\sumflat_{\substack{x\leq (-1)^kd\leq x+H\\
\sqrt{\alpha_g}|c_g(|d|)|> k^{-\delta}}}\sqrt{\alpha_g}|c_g(|d|)|\,+\sumflat_{\substack{x\leq (-1)^kd\leq x+H\\
\sqrt{\alpha_g}|c_g(|d|)|> k^{-\delta}}}\sqrt{\alpha_g}c_g(|d|)\right|\\
&\gg \frac H{\sqrt k \log X}.
\end{align*}
Similarly,
\begin{align*}
2\sumflat_{\substack{x\leq (-1)^kd\leq x+H\\
\sqrt{\alpha_g}c_g(|d|)< -k^{-\delta}}}\sqrt{\alpha_g}|c_g(|d|)|&=\left|\sumflat_{\substack{x\leq (-1)^kd\leq x+H\\
\sqrt{\alpha_g}|c_g(|d|)|> k^{-\delta}}}\sqrt{\alpha_g}|c_g(|d|)|\,-\sumflat_{\substack{x\leq (-1)^kd\leq x+H\\
\sqrt{\alpha_g}|c_g(|d|)|> k^{-\delta}}}\sqrt{\alpha_g}c_g(|d|)\right|\\
&\gg \frac H{\sqrt k \log X}.
\end{align*}
Thus we have shown that for $\gg_\eps K^2/(\log K)^{3/2+\eps}$ of the forms $g\in\mathcal S_K$ the short interval $[x,x+H]$, $x\sim X$, contains numbers $(-1)^k d_\pm$, with both $d_\pm$ odd fundamental discriminants, for which $\sqrt{\alpha_g}c_g(|d_+|)>k^{-\delta}$ and $\sqrt{\alpha_g}c_g(|d_-|)<-k^{-\delta}$, for $\gg X/(\log X)^{5/2}$ of the numbers $x\sim X$. As discussed above, this leads to 
\[ 
\gg \frac X{(\log X)^{5/2}H}\asymp\frac{K}{Y(\log X)^{23/2}}\]
real zeroes on both of the line segments $\delta_1$ and $\delta_2$, as claimed. \qed \\

\noindent Next we shall show how the above two propositions just applied follow from Lemmas \ref{second-moment-with-average} and \ref{fourth-moment-with-average}. We start by proving the first key proposition.

\begin{subsection}{Proof of Proposition \ref{Prop1}}
Recall that
\[\mathcal S_K=\bigcup_{k\sim K}B_{k+\frac12}^+.\]
Denote
\[S_{1,g}(x;H):=\left|\sumflat_{x\leq (-1)^kd\leq x+H}\sqrt{\alpha_g}c_g(|d|)\right|\]
and 
\[\mathcal T_{1,g}(X;H):=\#\left\{x\sim X:\,S_{1,g}(x;H)\geq \sqrt H \cdot k^{-1/2}(\log K)^3\right\}.\]
Choose $h$ to be a non-negative smooth function that is supported in the interval $[1/2,5/2]$ and is identically one in $[1,2]$. With the above notation the quantity we need to bound is by Markov's inequality
\begin{align}\label{application-of-Chebyshev}
&\sum_{k\sim K}\sum_{\substack{g\in B_{k+\frac12}^+\\|\mathcal T_{1,g}(X;H)|\geq X/(\log X)^3}} 1 \nonumber\\
&\leq\frac{(\log X)^3}{X}\sum_{k\in\Z}h\left(\frac{k}K\right)\sum_{g\in B_{k+\frac12}^+}|\mathcal T_{1,g}(X;H)| \nonumber\\
&=\frac{(\log X)^3}{X}\sum_{k\in\Z}h\left(\frac{k}K\right)\sum_{g\in B_{k+\frac12}^+}\sum_{\substack{x\sim X\\
|S_{1,g}(x;H)|\geq \sqrt H k^{-1/2}(\log K)^3}}1 \nonumber \\
&\ll \frac {K(\log X)^3}{H(\log K)^6X}\sum_{k\in\Z}h\left(\frac{k}K\right)\sum_{g\in B_{k+\frac12}^+}\alpha_g\sum_{x\sim X}\left|\sumflat_{x\leq (-1)^kd\leq x+H}c_g(|d|)\right|^2. 
\end{align}
By opening the absolute square the inner sum over $x\sim X$ can be rearranged into
\[ 
\sum_{x\sim X}\sumflat_{x\leq(-1)^kd_1\leq x+H}\sumflat_{x\leq(-1)^kd_2\leq x+H}c_g(|d_1|)c_g(|d_2|).\]
Let us first focus on the diagonal terms with $d_1=d_2$. In this case the total contribution to (\ref{application-of-Chebyshev}) is given by

\[\ll \frac {K(\log X)^3}{(\log K)^6X}\sum_{k\in\Z}h\left(\frac{k}K\right)\sum_{g\in B_{k+\frac12}^+}\alpha_g\sumflat_{(-1)^kd\sim X}|c_g(|d|)|^2.\]
Adding a smooth weight function that localises $(-1)^kd\sim X$ and applying Lemma \ref{second-moment-with-average} this is bounded by
\begin{align*}
&\ll K^2\frac {(\log X)^3}{(\log K)^6}\ll \frac{K^2}{(\log K)^3}, 
\end{align*}
as desired.   

Let us then focus on the off-diagonal corresponding to the terms with $d_1\neq d_2$. We apply Lemma \ref{half-int-Petersson} to see that 
\begin{align}\label{off-diag-estimate}
&\sum_{k\in\mathbb Z}h\left(\frac{k}K\right)\sum_{x\sim X}\sumflat_{x\leq(-1)^kd_1\leq x+H}\sumflat_{\substack{x\leq(-1)^kd_2\leq x+H\\
d_1\neq d_2}}\sum_{g\in B^+_{k+\frac12}}\alpha_g c_g(|d_1|)c_g(|d_2|) \nonumber\\
&=\frac{2\pi}3 e\left(-\frac18\right)\sum_{k\in\mathbb Z}h\left(\frac{k}K\right)\sum_{x\sim X}\sumflat_{x\leq(-1)^kd_1\leq x+H}\sumflat_{\substack{x\leq(-1)^kd_2\leq x+H\\
d_1\neq d_2}}\sum_{4|c}\frac{K_{k+\frac12}^+(|d_1|,|d_2|;c)}c \cdot e\left(-\frac k4\right) J_{k-\frac12}\left(\frac{4\pi\sqrt{|d_1d_2|}}c\right).
\end{align}
But using the uniform estimate (\ref{J-Bessel}), the Weyl bound (\ref{modified-Kloosterman-bound}) and noting that $|d_1|$, $|d_2|\asymp X$, the inner sum may be bounded as 
\begin{align*}
\sum_{4|c}\frac{K_{k+\frac12}^+(|d_1|,|d_2|;c)}c \cdot i^{k} J_{k-\frac12}\left(\frac{4\pi\sqrt{|d_1d_2|}}c\right)&\ll_\eps \sum_{4|c}\frac{c^{1/2+\eps}X^{1/2}}c\cdot\frac{\sqrt{|d_1d_2|}}c\cdot\frac1{\sqrt k}\left(\frac{2\pi e\sqrt{|d_1d_2|}}{ck}\right)^{k-1/2}\\
&\ll_\eps\frac{X^{3/2}}{\sqrt k}\left(\frac{2\pi e\sqrt{|d_1d_2|}}{k}\right)^{k-1/2}\sum_{4|c}\frac1{c^{1+k-\eps}}\\
&\ll\frac{X^{3/2}}{\sqrt k}\left(\frac{2\pi e\sqrt{|d_1d_2|}}{k}\right)^{k-1/2} .
\end{align*}
Now recall that $X\ll\sqrt K$ and $k\asymp K$ to see that the latter factor on the right-hand side decays exponentially. Thus estimating all the other sums in (\ref{off-diag-estimate}) trivially shows that the off-diagonal contribution is    
\begin{align*} 
\ll \frac K{HX}\cdot\frac{X^{3/2}}{\sqrt K}\cdot KXH^2\left(K^{-1/4}\right)^{K-1/2},
\end{align*}
which is negligible in the above range of $X$. This completes the proof. \qed 
\end{subsection} 

\begin{subsection}{Proof of Proposition \ref{Prop2}}

We start by deriving a lower bound for the weighted sum of the terms $|c_g(|d|)|$ on average over the forms $g\in\mathcal S_K$. Applying H\"older's inequality we have 
\begin{align*}
&\sum_{k\in\Z}h\left(\frac{k}K\right)\sum_{g\in B_{k+\frac12}^+}\alpha_g\sumflat_{(-1)^kd\sim X}|c_g(|d|)|^2\\
&\leq\left(\sum_{k\in\Z}h\left(\frac{k}K\right)\sum_{g\in B_{k+\frac12}^+}\sumflat_{(-1)^kd\sim X}\alpha_g^{1/2}\omega_g^{1/2}|c_g(|d|)|\right)^{2/3}\left(\sum_{k\sim K}\sum_{g\in B_{k+\frac12}^+}\sumflat_{(-1)^kd\sim X}\alpha_g^2\omega_g^{-1}|c_g(|d|)|^4\right)^{1/3},
\end{align*}
where $h$ is a smooth compactly supported function that minorises the characteristic function of the interval $[1,2]$ and satisfies $\widehat h(0)\neq 0$.

Using Lemmas \ref{second-moment-with-average} and \ref{fourth-moment-with-average} for the second and fourth moments of the Fourier coefficients respectively we obtain
\begin{align*}
\sum_{k\in\Z}h\left(\frac{k}K\right)\sum_{g\in B_{k+\frac12}^+}\sumflat_{(-1)^kd\sim X}\alpha_g^{1/2}\omega_g^{1/2}|c_g(|d|)|\gg\frac{XK}{(\log XK)^{1/2}}.
\end{align*}
On the other hand, by Lemma \ref{second-moment-with-average} and Waldspurger's formula we also have 
\begin{align*}
&\sum_{k\in\Z}h\left(\frac{k}K\right)\sum_{g\in B_{k+\frac12}^+}\sumflat_{\substack{(-1)^kd\sim X\\
L(1/2,f\otimes \chi_{d})>(\log XK)^2}}\alpha_g^{1/2}\omega_g^{1/2}|c_g(|d|)|\\
&<(\log XK)^{-1}\sum_{k\in\Z}h\left(\frac{k}K\right)\sum_{g\in B_{k+\frac12}^+}\sumflat_{\substack{(-1)^kd\sim X\\
L(1/2,f\otimes\chi_{d})>(\log XK)^2}}\alpha_g|c_g(|d|)|^2\\
&\ll \frac{XK}{\log XK}.
\end{align*}
In the first estimate we have used (\ref{connecting-normalisations}) in the form
\begin{align*}
\alpha_g^{1/2}\omega_g^{1/2}|c_g(|d|)|&=\omega_gL\left(\frac12,f\otimes\chi_d\right)^{1/2}\\
&=\frac{\log XK}{\log XK}\omega_g L\left(\frac12,f\otimes\chi_d\right)^{1/2}\\
&<(\log XK)^{-1}\omega_g L\left(\frac12,f\otimes\chi_d\right)\\
&=(\log XK)^{-1}\alpha_g|c_g(|d|)|^2
\end{align*}
when $L(1/2,f\otimes\chi_d)>(\log XK)^2$.

From this we infer the lower bound
\[\sum_{k\in\Z}h\left(\frac{k}K\right)\sum_{g\in B_{k+\frac12}^+}\sumflat_{\substack{(-1)^kd\sim X\\
L(1/2,f\otimes\chi_{d})\leq (\log XK)^{2}}}\alpha_g^{1/2}\omega_g^{1/2}|c_g(|d|)|\gg \frac{KX}{(\log XK)^{1/2}}.\]

Let us now define the set
\[\mathcal V_g:=\left\{x\sim X:\,\sumflat_{\substack{x\leq(-1)^kd\leq x+H\\
L(1/2,f\otimes\chi_{d})\leq (\log XK)^{2}}}\sqrt{\alpha_g}|c_g(|d|)|\geq\frac{H}{k^{1/2}\log X}\right\}.\]

From the work above it follows that
\begin{align*}
\frac{KX}{(\log XK)^{1/2}}&\ll\sum_{k\in\Z}h\left(\frac{k}K\right)\sum_{g\in B_{k+\frac12}^+}\sumflat_{\substack{(-1)^kd\sim X\\
L(1/2,f\otimes\chi_{d})\leq (\log XK)^{2}}}\alpha_g^{1/2}\omega_g^{1/2}|c_g(|d|)|\\
&\ll\frac1H\sum_{k\in\Z}h\left(\frac{k}K\right)\sum_{g\in B_{k+\frac12}^+}\sum_{x\sim X}\sumflat_{\substack{x\leq (-1)^kd\leq x+H\\
L(1/2,f\otimes\chi_{d})\leq (\log XK)^{2}}}\alpha_g^{1/2}\omega_g^{1/2}|c_g(|d|)|\\
&=\frac1H\sum_{k\in\Z}h\left(\frac{k}K\right)\sum_{g\in B_{k+\frac12}^+}\left(\sum_{x\in\mathcal V_g}+\sum_{x\not\in\mathcal V_g}\right)\sumflat_{\substack{x\leq (-1)^kd\leq x+H\\
L(1/2,f\otimes\chi_{d})\leq (\log XK)^{2}}}\alpha_g^{1/2}\omega_g^{1/2}|c_g(|d|)|\\
&\leq\frac1H\sum_{k\in\Z}h\left(\frac{k}K\right)\sum_{g\in B_{k+\frac12}^+}\left(\omega_g(\log XK)H\left|\mathcal V_g\right|+\omega_g^{1/2}\frac{XH}{k^{1/2}\log X}\right),
\end{align*}
where we have used the relation $\omega_g^{1/2}\alpha_g^{1/2}|c_g(|d|)|=\omega_g\sqrt{L(1/2,f\otimes\chi_d)}$ in the last estimate. Using an easy estimate $\sum_{g\in B_{k+\frac12}^+}\omega_g^{1/2}\ll\sqrt k$ (which follows from the Cauchy--Schwarz inequality and (\ref{weights-average})) we conclude that
\begin{align}\label{lower-bound}
\sum_{k\in\Z}h\left(\frac{k}K\right)\sum_{g\in B_{k+\frac12}^+}\omega_g\left|\mathcal V_g\right|\gg \frac{KX}{(\log XK)^{3/2}}.
\end{align}
To remove the harmonic weights $\omega_g$ note that by \cite[(8.18)]{Balkanova-Frolenkov2021} we have 
\begin{align*}
\sum_{g\in B_{k+\frac12}^+}|\mathcal V_g|\gg K\sum_{f\in S_{2k}(1)}\omega_f|\mathcal V_g|L(1,\text{sym}^2f).
\end{align*}
As there exists $A>0$ so that $L(1,\text{sym}^2f)\gg (\log\log K)^{-A}$ for all but $O(K^\eps)$ Hecke eigenforms $f\in\mathcal B_k$ (see \cite{Lau-Wu2006}) and we have the Hoffstein--Lockhart bounds $1/\log k\ll L(1,\text{sym}^2f)\ll\log k$, it follows from (\ref{lower-bound}) that
\begin{align}\label{V_g-on-av}
\sum_{k\in\Z}h\left(\frac{k}K\right)\sum_{g\in B_{k+\frac12}^+}\left|\mathcal V_g\right|\gg_\eps \frac{XK^2}{(\log XK)^{3/2+\eps}}.
\end{align}

\noindent Let us introduce the set
\[\mathcal U:=\left\{g\in\mathcal S_K:\,|\mathcal V_g|\gg\frac{X}{(\log X)^{5/2}}\right\}.\]
Now from (\ref{V_g-on-av}) we deduce that 
\begin{align*}
\frac{XK^2}{(\log XK)^{3/2+\eps}}&\ll_\eps\sum_{g\in\mathcal U}\left|\mathcal V_g\right|+\sum_{g\in \mathcal S_K\setminus\mathcal U}\left|\mathcal V_g\right|\\
&\ll |\mathcal U|X+K^2\cdot\frac X{(\log X)^{5/2}}
\end{align*}
from which we infer the lower bound 
\[|\mathcal U|\gg_\eps\frac{K^2}{(\log XK)^{3/2+\eps}}. \]

\noindent Hence we have shown that for $\gg_\eps K^2/(\log K)^{3/2+\eps}$ of the forms $g\in \mathcal S_K$ we have
\begin{align*}
&\#\left\{x\sim X\,:\,\sumflat_{x\leq (-1)^kd\leq x+H}\sqrt{\alpha_g}|c_g(|d|)|\geq\frac H{k^{1/2}\log X}\right\}\\
&\geq\#\left\{x\sim X\,:\,\sumflat_{\substack{x\leq (-1)^kd\leq x+H\\
L(1/2,f\otimes\chi_d)\leq (\log XK)^2}}\sqrt{\alpha_g}|c_g(|d|)|\geq\frac H{k^{1/2}\log X}\right\}\\
&\gg\frac X{(\log X)^{5/2}},
\end{align*}
which is what we wanted to prove. \qed

\end{subsection}

\section{Proof of Theorem \ref{second-main-theorem}}

\subsection{Proof of Proposition \ref{non-vanishing}}
\noindent First we explain how to use Lemmas \ref{second-moment} and \ref{fourth-moment} to deduce Proposition \ref{non-vanishing}. Recall the definition of the Iwaniec--Sarnak mollifier $\mathcal M_{f,d}$ from Section $2$. Let $\theta>0$ be arbitrarily small, but fixed. Also, let $h$ be a compactly supported smooth function specified later. Note that by (\ref{connecting-normalisations}) the constraint $\sqrt{\alpha_g}|c_g(|d|)|> k^{-1/2-\theta}$ is equivalent to $\omega_f L(1/2,f\otimes\chi_d)> k^{-1-2\theta}$. Using the Cauchy--Schwarz inequality we have 
\begin{align*}
&\left(\sum_{k\in\Z}h\left(\frac{2k}K\right)\sum_{\substack{f\in \mathcal B_k\\
\omega_f L(1/2,f\otimes\chi_d)> k^{-1-2\theta}}}
\omega_f L\left(\frac12,f\otimes\chi_{d}\right)\mathcal M_{f,d}\right)^2\\
&\leq\left(\sum_{k\in\Z}h\left(\frac{2k}K\right)\sum_{\substack{f\in\mathcal B_k\\
\omega_f L(1/2,f\otimes\chi_d)> k^{-1-2\theta}}}
 \omega_f\right)\left(\sum_{k\in\Z}h\left(\frac{2k}K\right)\sum_{\substack{f\in\mathcal B_k\\
\omega_f L(1/2,f\otimes\chi_d)> k^{-1-2\theta}}}\omega_f L\left(\frac12,f\otimes\chi_{d}\right)^2\mathcal M_{f,d}^2 \right).
\end{align*}
First we note that 
\begin{align*}
&\sum_{k\in\Z}h\left(\frac{2k}K\right)\sum_{\substack{f\in \mathcal B_k\\
\omega_f L(1/2,f\otimes\chi_d)> k^{-1-2\theta}}}
\omega_f L\left(\frac12,f\otimes\chi_{d}\right)\mathcal M_{f,d}\\
&=\sum_{k\in\Z}h\left(\frac{2k}K\right)\sum_{f\in \mathcal B_k}
\omega_f L\left(\frac12,f\otimes\chi_{d}\right)\mathcal M_{f,d}+O\left(\sum_{k\in\Z}h\left(\frac{2k}K\right)\sum_{\substack{f\in \mathcal B_k\\
\omega_f L(1/2,f\otimes\chi_d)\leq k^{-1-2\theta}}}
\omega_f L\left(\frac12,f\otimes\chi_{d}\right)\mathcal M_{f,d}\right)\\
&=\sum_{k\in\Z}h\left(\frac{2k}K\right)\sum_{f\in \mathcal B_k}
\omega_f L\left(\frac12,f\otimes\chi_{d}\right)\mathcal M_{f,d}+O\left(K^{-1-2\theta}\sum_{k\in\Z}h\left(\frac{2k}K\right)\sum_{f\in\mathcal B_k}|\mathcal M_{f,d}|\right)\\
&=\sum_{k\in\Z}h\left(\frac{2k}K\right)\sum_{f\in \mathcal B_k}
\omega_f L\left(\frac12,f\otimes\chi_{d}\right)\mathcal M_{f,d}+O_\eps\left(K^{1-2\theta+\eps}\right)\\
&\sim \frac K2\int\limits_0^\infty h(t)\,\mathrm d t,
\end{align*}
where in the last step Lemma \ref{second-moment} was used and the penultimate estimate follows from the easy bound\footnote{Note that the normalisation in \cite{Iwaniec-Sarnak2000} differs slightly from ours.}
\[ 
\sum_{k\in\Z}h\left(\frac{2k}K\right)\sum_{f\in\mathcal B_k}|\mathcal M_{f,d}|\ll_\eps K^{2+\eps},\]
which is a direct consequence of the penultimate estimate in \cite[p. 164]{Iwaniec-Sarnak2000} and the Cauchy--Schwarz inequality.
   
On the other hand, by Lemma \ref{fourth-moment} and using non-negativity of the summands we have
\begin{align*}
&\sum_{k\in\Z}h\left(\frac{2k}K\right)\sum_{\substack{f\in\mathcal B_k\\
\omega_f L(1/2,f\otimes\chi_d)> k^{-1-2\theta}}}\omega_f L\left(\frac12,f\otimes\chi_{d}\right)^2\mathcal M_{f,d}^2\\
&\leq \sum_{k\in\Z}h\left(\frac{2k}K\right)\sum_{f\in\mathcal B_k}\omega_f L\left(\frac12,f\otimes\chi_{d}\right)^2\mathcal M_{f,d}^2\\
&\sim K\left(\int\limits_0^\infty h(t)\,\mathrm d t\right)\left(1+\frac{\log |d|K}{\log L} \right).
\end{align*}

Thus, by choosing the length $L$ of the mollifier to be as long as possible i.e. $L=|d|^{-1}K(\log K)^{-20}$, 
\begin{align*}
&\sum_{k\in\Z}h\left(\frac{2k}K\right)\sum_{\substack{f\in \mathcal B_k\\
\omega_f L(1/2,f\otimes\chi_d)> k^{-1-2\theta}}}\omega_f\\
&\geq\frac{K^2}{4K}\cdot\frac1{1+\frac{\log |d|K}{\log L}}\left(\int\limits_0^\infty h(t)\,\mathrm d t\right)+o(K) \nonumber\\
&\geq\frac K4\left(1-\frac{\log |d|}{2\log K-20\log\log K}-\frac{\log K}{2\log K-20\log\log K}+o_{K\rightarrow\infty}(1)\right)\left(\int\limits_0^\infty h(t)\,\mathrm d t\right) \nonumber \\
&>\frac K8\left(1-\frac{\log |d|}{\log K}+o_{K\rightarrow\infty}(1)\right)\left(\int\limits_0^\infty h(t)\,\mathrm d t\right)
\end{align*}
for all fundamental discriminants $|d|\leq K^c$ (with some $c>0$). As
\[ 
\sum_{k\in\Z}h\left(\frac{2k}K\right)\sum_{f\in\mathcal B_k}\omega_f\sim\frac K2\int\limits_0^\infty h(t)\,dt\]
this in particular implies
\[ 
\sum_{k\in\Z}h\left(\frac{2k}K\right)\sum_{\substack{f\in \mathcal B_k\\
\omega_f L(1/2,f\otimes\chi_d)> k^{-1-2\theta}}}\omega_f>\left(\frac14-\frac\eps2\right)\sum_{k\in\Z}h\left(\frac{2k}K\right)\sum_{f\in\mathcal B_k}\omega_f\]
for large enough $K$. 

Now choose the weight function $h$ so that it is supported in the interval $[2,4]$ and is identically one in $[2+\nu,4-\nu]$, where $\nu>0$ is fixed and sufficiently small in terms of $\eps>0$. The harmonic weights $\omega_f$ may be removed again by using \cite[Lemma 8.9.]{Balkanova-Frolenkov2021} (similarly as in \cite{Iwaniec-Sarnak2000}, in particular see the comments on p. $165$) to get the same proportion of the event $\omega_fL(1/2,f\otimes\chi_d)>k^{-1-2\theta}$ for the natural average:
\begin{align}\label{non-zero-Fourier-coeff}
\sum_{k\sim K}\sum_{\substack{f\in \mathcal B_k\\
\omega_f L(1/2,f\otimes\chi_d)> k^{-1-2\theta}}}1 &\geq \sum_{k\in\Z}h\left(\frac{2k}K\right)\sum_{\substack{f\in \mathcal B_k \\
\omega_f L(1/2,f\otimes\chi_d)> k^{-1-2\theta}}}1 \nonumber\\
&>\left(\frac1{4}-\frac\eps2\right)\sum_{k\in\Z}h\left(\frac{2k}K\right)\sum_{f\in\mathcal B_k}1 \nonumber\\
&\geq\left(\frac14-\frac\eps2\right)(1-\nu)\#\mathcal S_K \nonumber \\
&\geq \left(\frac14-\eps\right)\#\mathcal S_K.
\end{align}
for large enough $K$. This completes the proof of Proposition \ref{non-vanishing}.

\subsection{Sign changes of Fourier coefficients along squares}

The final thing to do is to use Proposition \ref{non-vanishing} to complete the proof of Theorem \ref{second-main-theorem}. Recall that our aim is to produce many pairs of odd integers for which (\ref{real-zero}) holds. For this we seek such pairs in short intervals and the idea here is to connect short sums of the Fourier coefficients to long sums using the results of \cite{Matomaki-Radziwill2016}. Let $\eta>0$ be as in Proposition \ref{Fourier-coeff}. Given $g\in \mathcal S_K$, fix any odd fundamental discriminant $d$ and define the multiplicative function $h_g:\N\longrightarrow\{-1,0,1\}$ by setting
\[
h_g(p^\ell):=\begin{cases}
\text{sgn}(c_g(|d|)^{-1}c_g(|d|p^{2\ell})) & \text{if }|c_g(|d|)^{-1}c_g(|d|p^{2\ell})|\geq p^{-\ell\eta/2}\,\text{and }p>2\\
0 & \text{otherwise}
\end{cases} 
\] 
Here $\text{sgn}(x)$ is the sign of $x\in\R$ and we interpret $c_g(|d|)^{-1}=0$ if $c_g(|d|)=0$.

Note that if $h_g(m)\neq 0$, then $m$ is odd, $c_g(|d|)\neq 0$, $|c_g(|d|m^2)|\geq |c_g(|d|)|m^{-\eta/2}$, and finally $h_g(m)=\text{sgn}(c_g(|d|)^{-1}c_g(|d|m^2))$. These follow from the fact that $m\mapsto c_g(|d|m^2)$ satisfy the multiplicative property $c_g(|d|)c_g(|d|m^2n^2)=c_g(|d|m^2)c_g(|d|n^2)$ (when $(m,n)=1$) \cite[(1.18)]{Shimura1973} for any fixed fundamental discriminant $d$. From this it follows that the map $m\mapsto c_g(|d|)^{-1}c_g(|d|m^2)$ defines a multiplicative function when $c_g(|d|)\neq 0$. 

Let $\theta>0$ be a fixed quantity chosen later and define 
\[
S_{K,d}:=\bigcup_{k\sim K}\{g\in B_{k+\frac12}^+:\,\sqrt{\alpha_g}|c_g(|d|)|> k^{-1/2-\theta}\}.
\] 
To prove Theorem \ref{second-main-theorem} we adapt the argument of \cite[Section 3]{Lester-Matomaki-Radziwill2018}. We start with a key auxiliary result. 

\begin{lemma}\label{final-key-lemma}
Let $\varepsilon>0$ be arbitrary but fixed, and $d\ll 1$ be a fixed fundamental discriminant\footnote{Recall that the function $h_g$ depends on $d$.}. Then there exists $X_0(\varepsilon)$ such that for $X_0<X<K$ we have, for all but at most $\varepsilon\# S_{K,d}$ forms $g\in S_{K,d}$,
\[ 
\left|\frac1X\sum_{n\sim X}|h_g(n)|-\frac1X\left|\sum_{n\sim X}h_g(n)\right|\right|\gg_{\eta,\varepsilon} 1.\]
\end{lemma} 

\begin{proof}

It suffices to show the following two assertions: 

\begin{enumerate}
\item For all but at most $(\varepsilon/2)\#S_{K,d}$ of the forms $g\in S_{K,d}$ we have 
\[ 
\frac1X\sum_{n\sim X}|h_g(n)|\gg_{\eta,\varepsilon}1,\]
where the implicit constant is independent of $g$. \\
\item For all but at most $(\varepsilon/2)\#S_{K,d}$ of the forms $g\in S_{K,d}$ we have
\[ 
\frac1X\sum_{n\sim X}h_g(n)=o(1).\]
\end{enumerate}
\noindent Towards (1) we argue as follows. Let us first assume that the fixed fundamental discriminant $d$ is positive. Then note that by the identity (\ref{Fourier-coeff-rel}), the triangle inequality, and Lemma \ref{Murty-Sinha} we have
\begin{align*}
\sum_{\substack{k\sim K\\
k\equiv 0\,(2)}}\sum_{\substack{g\in B_{k+\frac12}^+\\
\sqrt{\alpha_g}|c_g(|d|)|>k^{-1/2-\theta}}}\sum_{\substack{p\leq X\\
|c_g(|d|)^{-1}c_g(|d|p^2)|<p^{-\eta/2}}}\frac1p &\ll\sum_{\substack{k\sim K\\
k\equiv 0\,(2)}}\sum_{f\in \mathcal B_k}\sum_{\substack{p\leq X\\
|\lambda_f(p)-\chi_d(p)/\sqrt p|<p^{-\eta/2}}}\frac1p \\
&\leq \sum_{\substack{k\sim K\\
k\equiv 0\,(2)}}\sum_{f\in \mathcal B_k}\sum_{\substack{p\leq X\\
|\lambda_f(p)|<p^{-\eta/2}+1/\sqrt p}}\frac1p \\
&\ll\sum_{\substack{k\sim K\\
k\equiv 0\,(2)}} \#\mathcal B_k\sum_{p\leq X}\left(p^{-3/2}+p^{-1-\eta/2}+\frac{\log p}{p\log k}\right) \\
&\ll_\eta\sum_{\substack{k\sim K\\
k\equiv 0\,(2)}}\#\mathcal B_k\\
&\ll\sum_{\substack{k\sim K\\
k\equiv 0\,(2)}}\sum_{\substack{g\in B_{k+\frac12}^+\\
\sqrt{\alpha_g}|c_g(|d|)|>k^{-1/2-\theta}}}1,
\end{align*}
where in the last step we have used (\ref{non-zero-Fourier-coeff}) (together with (\ref{connecting-normalisations}) and the fact that $c_g(|d|)=0$ for positive $d$ when $k$ is odd), and the fact that $\#\mathcal B_k\asymp k$.

Hence there is a positive constant $C_1$ depending only at most on $\eta$ and $\theta$ such that for any given $\varepsilon>0$,
\begin{align*}
\sum_{\substack{p\leq X \\
h_g(p)=0}}\frac1p\leq\frac {C_1}{\varepsilon}
\end{align*}
for all but at most $(\varepsilon/2)\# I_{K,d}$ forms $g\in I_{K,d}$, where
\begin{align*}
I_{K,d}:=\bigcup_{\substack{k\sim K\\
k\equiv 0\,(2)}}\{g\in B_{k+\frac12}^+:\,\sqrt{\alpha_g}|c_g(|d|)|>k^{-1/2-\theta}\}.
\end{align*}
Similarly, for a fixed negative fundamental discriminant $d$, one shows that there is a positive constant $C_2$ depending only at most on $\eta$ and $\theta$ such that for any given $\varepsilon>0$,
\begin{align*}
\sum_{\substack{p\leq X \\
h_g(p)=0}}\frac1p\leq\frac{C_2}{\varepsilon}
\end{align*}
for all but at most $(\varepsilon/2)\# J_{K,d}$ forms $g\in J_{K,d}$, where
\begin{align*}
J_{K,d}:=\bigcup_{\substack{k\sim K\\
k\equiv 1\,(2)}}\{g\in B_{k+\frac12}^+:\,\sqrt{\alpha_g}|c_g(|d|)|>k^{-1/2-\theta}\}.
\end{align*} 
Note that for any given $d$, $S_{K,d}=I_{K,d}\cup J_{K,d}$ with one of the sets $I_{K,d},\, J_{K,d}$ being empty (as $c_g(|d|)=0$ when $(-1)^kd<0)$. Now the first assertion follows from \cite[Theorem 2]{Hildebrand1987}. 

For the second assertion, observe that, for a fixed positive fundamental discriminant $d$, 
\begin{align}\label{identity-towards-Halasz}
\sum_{\substack{p\leq X\\
h_g(p)=-1}}\frac1p=\sum_{\substack{p\leq X\\
c_g(|d|)^{-1}c_g(|d|p^2)<0}}\frac1p-\sum_{\substack{p\leq X\\
-p^{-\eta/2}<c_g(|d|)^{-1}c_g(|d|p^2)<0}}\frac 1p.
\end{align}
Recalling (\ref{Fourier-coeff-rel}) and using the fact that $\chi_d(p)\in\{-1,0,1\}$ it follows that the first sum on the right-hand side of (\ref{identity-towards-Halasz}) is
\begin{align*}
&\geq \sum_{\substack{p\leq X\\
\lambda_f(p)<-1/\sqrt p}}\frac1p \\
&\geq \sum_{\substack{p\leq X\\
\lambda_f(p)<0}}\frac1p-\sum_{\substack{p\leq X\\
-1/\sqrt p<\lambda_f(p)<0}}\frac1p.
\end{align*}
Arguing identically as in \cite[p. 1610]{Lester-Matomaki-Radziwill2018} we have that
\begin{align*}
\sum_{\substack{p\leq X\\
\lambda_f(p)<0}}\frac1p\geq\frac{1+o(1)}8\log\log X+\sum_{\log X\leq p\leq X^{1/1000}}\frac{\lambda_f(p^2)-2\lambda_f(p)}p
\end{align*} 
and using the large sieve inequality of Lemma \ref{large-sieve-2} the latter sum on the right-hand side contributes $o(\log\log X)$ for almost all forms $f\in\cup_{k\sim K\,\text{even}}\mathcal B_k$. 

To summarise, we have deduced that 
\begin{align*}
\sum_{\substack{p\leq X\\
h_g(p)=-1}}\frac1p\geq\frac{(1+o(1))}{8}\log\log X-\sum_{\substack{p\leq X\\
-1/\sqrt p<\lambda_f(p)<0}}\frac1p-\sum_{\substack{p\leq X\\
-p^{-\eta/2}<c_g(|d|)^{-1}c_g(|d|p^2)<0}}\frac 1p
\end{align*}
for almost all forms $g\in\cup_{k\sim K\,\text{even}}B_{k+1/2}^+$. But from the arguments used to establish the first assertion above we know that there is an absolute constant $C_3$ so that
\[ 
\sum_{\substack{p\leq X\\
-p^{-\eta/2}<c_g(|d|)^{-1}c_g(|d|p^2)<0}}\frac 1p\leq\frac{C_3}{\varepsilon}\]
for all but $(\varepsilon/4)\# I_{K,d}$ of the forms $g\in I_{K,d}$. A similar computation also shows that there exists an absolute constant $C_4$ so that
\[ 
\sum_{\substack{p\leq X\\
-1/\sqrt p<\lambda_f(p)<0}}\frac1p\leq\frac{C_4}{\varepsilon}
\]
for all but $(\varepsilon/4)\# I_{K,d}$ of the forms $g\in I_{K,d}$ (recall here that $f\in\mathcal B_k$ corresponds to some $g\in B_{k+\frac12}^+$ under the Shimura correspondence). Thus for all but at most $(\eps/2)\#I_{K,d}$ of the forms $g\in I_{K,d}$ we have
\[
\sum_{\substack{p\leq X\\
h_g(p)=-1}}\frac1p\geq\frac{(1+o(1))}{8}\log\log X. 
\]
Similarly one shows that the same estimate holds for all but $(\varepsilon/2)\# J_{K,d}$ of the forms $g\in J_{K,d}$ when the fixed fundamental discriminant $d$ is negative. This finishes the proof of (2) by Hal\'asz's theorem (Lemma \ref{Halasz}) by recalling that one of the sets $I_{K,d},\,J_{K,d}$ is empty for any given $d$. This concludes the proof of the lemma. 

\end{proof}

\noindent To finish the proof of Theorem \ref{second-main-theorem} we argue as follows. Set now $\theta=\eta/4$, and let $\eps>0$ be arbitrarily small, but fixed.  Note that by definition, with $|d|\ll 1$ being fixed and choosing $X\asymp\sqrt{K/|d|Y}$, a sign change of $h_g$ on $[X,2X]$ means the existence of odd natural numbers $m$ and $n$ at most $2X$ so that
\[ 
c_g(|d|m^2)\leq-|c_g(|d|)|m^{-\eta/2}<|c_g(|d|)|n^{-\eta/2}\leq c_g(|d|n^2).
\]
Multiplying both sides by $\sqrt{\alpha_g}$ we require 
\[ 
\sqrt{\alpha_g}c_g(|d|m^2)\leq-\sqrt{\alpha_g}|c_g(|d|)|m^{-\eta/2}<\sqrt{\alpha_g}|c_g(|d|)|n^{-\eta/2}\leq \sqrt{\alpha_g}c_g(|d|n^2).
\]
But if $\sqrt{\alpha_g}|c_g(|d|)|> k^{-1/2-\eta/4}$ (which holds for at least $(1/4-\eps/4)\#\mathcal S_K$ of the forms $g\in\mathcal S_K$ by Proposition \ref{non-vanishing}) this implies, as e.g. $n\leq 2X\ll \sqrt k$, 
\[
\sqrt{\alpha_g}c_g(|d|m^2)<-k^{-\delta}<k^{-\delta}<\sqrt{\alpha_g}c_g(|d|n^2),
\]
where $\delta=1/2+\eta/2$. As discussed in the beginning of Section 9, this guarantees an existence of a zero on both of the lines $\delta_1$ and $\delta_2$.

Thus it suffices to exhibit sign changes of the multiplicative function $h_g$. Let us first fix a positive odd fundamental discriminant with $d_1\ll 1$ (one can e.g. take $d_1=5$) and denote the corresponding function $h_g$ by $h_g^{(1)}$. Suppose $H=H(\eta,\varepsilon)$ is sufficiently large fixed constant and $X_0(\eta,\varepsilon)<X\asymp\sqrt{K/|d_1|Y}$. By Proposition \ref{non-vanishing} we know that $\sqrt{\alpha_g}|c_g(|d_1|)|>k^{-1/2-\eta/4}$ for at least $(1/4-\eps/4)\#\mathcal S_K$ of the forms $g\in\mathcal S_K$. Then using Lemmas \ref{MR} and \ref{final-key-lemma} we have that for at least $(1-\varepsilon/4)\cdot(1/4-\eps/4)\#\mathcal S_K\geq(1/4-\varepsilon/2)\#\mathcal S_K$ of the forms $g\in \mathcal S_K$,
\begin{align*}
&\frac1H\sum_{x\leq n\leq x+H}|h_g^{(1)}(n)|-\frac1H\left|\sum_{x\leq n\leq x+H}h_g^{(1)}(n)\right|\\
&=\frac1X\sum_{n\sim X}|h_g^{(1)}(n)|-\frac1X\left|\sum_{n\sim X}h_g^{(1)}(n)\right|+O\left((\log H)^{-1/200}\right)\\
&\gg_{\eta,\varepsilon}1
\end{align*}
for all $x\sim X$ outside an exceptional set of size at most $CX(\log H)^{-1/100}$ for some absolute constant $C>1$. This implies that for each such form $g\in S_K$ there exist numbers $X\leq x_1<\cdots <x_N\leq 2X$ with $x_{j+1}-x_j>H$ and $N\geq\frac1{10}\cdot\frac XH$ such that
\[\left|\frac1H\sum_{x_j\leq n\leq x_j+H}|h_g^{(1)}(n)|-\frac1H\left|\sum_{x_j\leq n\leq x_j+H}h_g^{(1)}(n)\right|\right|\gg 1\]
for each $j=1,...,N$. We conclude that each interval $[x_j,x_j+H]$ yields a sign change of $h_g^{(1)}(n)$ for at least $(1/4-\varepsilon/2)\#\mathcal S_K$ of the forms $g\in\mathcal S_K$ above. As discussed in the introduction, this leads to
\[\gg\frac XH\gg\sqrt{\frac KY}\]
zeroes on both of the lines $\delta_1$ and $\delta_2$ with $\Im(z)\geq Y$ for at least $(1/4-\eps/2)\#\mathcal S_K$ of the forms $g\in\mathcal S_K$.

A similar argument shows that fixing a negative odd fundamental discriminant with $d_2$ with $|d_2|\ll 1$ (e.g. $d_2=-3$), the corresponding function $h_g^{(2)}$ has $\gg\sqrt{K/Y}$ sign changes for at least $(1/4-\eps/2)\#\mathcal S_K$ of the forms $g\in\mathcal S_K$. But as at most one of the inequalities $\sqrt{\alpha_g}|c_g(|d_1|)|>k^{-1/2-\eta/4}$ and $\sqrt{\alpha_g}|c_g(|d_2|)|>k^{-1/2-\eta/4}$ can hold for a given $g\in B_{k+\frac12}^+$ (as $d_1d_2<0$ and $c_g(|d|)=0$ whenever $(-1)^k d<0$) it follows that the subsets of $\mathcal S_K$ for which $h_g^{(1)}$ and $h_g^{(2)}$ exhibit sign changes are disjoint. Thus we get $\gg\sqrt{K/Y}$ real zeroes on both of the line segments $\delta_1$ and $\delta_2$ for at least $(1/2-\eps)\#\mathcal S_K$ forms in $\mathcal S_K$. This completes the proof. \qed



\bibliography{Half_integral_real_zeros}
\bibliographystyle{plain}

\end{document}